\documentclass[12pt, reqno]{amsart}
\usepackage{amsmath, amsthm, amscd, amsfonts, amssymb, graphicx, color}
\usepackage[bookmarksnumbered, colorlinks, plainpages]{hyperref}
\hypersetup{colorlinks=true,linkcolor=red, anchorcolor=green, citecolor=cyan, urlcolor=red, filecolor=magenta, pdftoolbar=true}
\input{mathrsfs.sty}
\usepackage{xcolor}
\usepackage{xcolor}
\usepackage{ulem}
\newtheorem{theorem}{Theorem}[section]
\newtheorem{lemma}[theorem]{Lemma}
\newtheorem{prop}[theorem]{Proposition}
\newtheorem{cor}[theorem]{Corollary}
\theoremstyle{definition}
\newtheorem{definition}[theorem]{Definition}
\newtheorem{example}[theorem]{Example}

\newtheorem{conjecture}[theorem]{Conjecture}

\theoremstyle{remark}
\newtheorem{remark}[theorem]{\bf{Remark}}
\numberwithin{equation}{section}
\begin{document}

	\title[On directional preservation of orthogonality]{On directional preservation of orthogonality and its application to isometries}
	
		\author[Manna,  Mandal, Paul and Sain  ]{Jayanta Manna, Kalidas Mandal,  Kallol Paul and Debmalya Sain }

	\newcommand{\acr}{\newline\indent}

	\address[Manna]{Department of Mathematics\\ Jadavpur University\\ Kolkata 700032\\ West Bengal\\ INDIA}
	\email{iamjayantamanna1@gmail.com}
	
	\address[Mandal]{Department of Mathematics\\ Jadavpur University\\ Kolkata 700032\\ West Bengal\\ INDIA}
	\email{kalidas.mandal14@gmail.com}
	
	\address[Paul]{Department of Mathematics\\ Jadavpur University\\ Kolkata 700032\\ West Bengal\\ INDIA}
	\email{kalloldada@gmail.com}
	
		\address[Sain]{Department of Mathematics\\ Indian Institute of Information Technology, Raichur\\ Karnataka 584135 \\INDIA}
	\email{saindebmalya@gmail.com}
	\thanks{Jayanta Manna would like to thank UGC, Govt. of India for the financial support
		in the form of Junior Research Fellowship under the mentorship of Professor Kallol Paul. The research of Dr Kalidas Mandal and 
		Professor Kallol Paul is supported by CRG Project bearing File no. CRG/2023/00716 of DST-SERB, Govt.
		of India.}

	
	\begin{abstract}
		We study the local preservation of Birkhoff-James orthogonality by linear operators between  normed linear spaces, at a point and in a particular direction. We obtain a complete characterization of the same, which allows us to present refinements of the local preservation of orthogonality explored earlier. We also study the directional preservation of orthogonality with respect to certain special subspaces of the domain space, and apply the results towards identifying the isometries on a polyhedral  normed linear space. In particular, we obtain refinements of the Blanco-Koldobsky-Turn\v{s}ek Theorem for polyhedral  normed linear spaces, including $ \ell_{\infty}^{n}, \ell_{1}^{n}. $

	\end{abstract}

	\maketitle

	\section{Introduction.}
	
	Identifying the isometries on a  normed linear space is one of the fundamental goals in functional analysis and operator theory. The classical Blanco-Koldobsky-Turn\v{s}ek Theorem \cite{BT06,K93} characterizes the isometries as the norm one linear operators preserving Birkhoff-James orthogonality. In light of this powerful result that deals with global preservation of orthogonality, it is interesting to consider local versions of the same. Such a program was recently initiated in \cite{SMP24}, where refinements of the Blanco-Koldobsky-Turn\v{s}ek Theorem were obtained in several particular cases, including the Euclidean spaces and the finite-dimensional sequence spaces $ \ell_{\infty}^{n}, \ell_{1}^{n}. $ The purpose of the present article is to conduct a finer study of the local preservation of orthogonality, leading to further refinements of some results obtained earlier in \cite{SMP24}.\\

	Let us now introduce the relevant terminologies and the notations to be used throughout this paper. We use the symbols  $ \mathbb{X}, \mathbb{Y}$  to denote  normed linear spaces.  We will be working with \textit{real} normed linear spaces, without mentioning it any further. The dual space of $\mathbb{X}$ is denoted by $\mathbb{X}^*.$ Let
	$B_{\mathbb{X}} = \{x\in \mathbb{X} : \|x\|\leq 1\}$ and $S_{\mathbb{X}} = \{x\in \mathbb{X} : \|x\|=1\}$ be the  unit ball and the unit sphere of $\mathbb{X}$, respectively. Let $\mathbb{L}(\mathbb{X}, \mathbb{Y})$ denote the space of all bounded linear operators from $ \mathbb{X} $ to $ \mathbb{Y}, $ endowed with the usual operator norm.  Whenever $ \mathbb{X} = \mathbb{Y}, $ we simply write $ \mathbb{L}(\mathbb{X}, \mathbb{Y}) = \mathbb{L}(\mathbb{X}). $ An operator $ T \in \mathbb{L}(\mathbb{X}, \mathbb{Y}) $ is said to be an isometry if $ \| Tx \| = \| x \|,$ for all $ x \in \mathbb{X}.$ The convex hull of a nonempty set $S\subset \mathbb{X}$ is denoted by $co(S).$ For a nonempty convex subset $A \subset \mathbb{X},$ $ x \in A$  is said to be an  extreme point of $A,$ if $x =(1-t)y +tz,$ for some $ t \in(0,1)$ and $y, z\in A$ imply that $x = y = z.$ The set of all extreme points of $A$ is denoted by $Ext~A.$ For a set $A\subset \mathbb{X},$ $A^c$ denotes the complement of $A,$ i.e., $A^c=\mathbb{X}\setminus A.$ For a non-zero element $x \in \mathbb{X},$ we denote the collection of all support functionals at $x$ as $J(x),$ i.e., $ J(x) = \{ f \in S_{\mathbb{X}^*} : f(x) = \| x \| \}. $   A non-zero element  $ x \in \mathbb{X} $ is said to be a $ k $-smooth point (or, that the order of smoothness of $ x $ is $k$)  if $dim\,span \,J(x)=k.$ Moreover, $ 1 $-smooth points are simply called smooth points. Thus, an element $ x \in \mathbb{X} $ is smooth if and only if there exists a unique support functional at $x.$ 
	For more information on smooth and $k$-smooth points in a  normed linear space and some of their applications, the readers are referred to \cite{DMP22,KS05,LR07,MPD22,MP20,PSG16,SPMR20,W18}. Given any two elements $x, y \in \mathbb{X}, $ we say that $x$ is Birkhoff-James orthogonal \cite{B35,J47}  to $y,$ written as $ x \perp_B y,$ if $ \| x + \lambda y \| \geq \|x\|$ for all scalars $\lambda. $  From James Characterization of Birkhoff-James orthogonality \cite[Cor. 2.2]{J47}, it follows that for any $x, y \in \mathbb{X}, $ $x\perp_B y$ if and only if there exists $f \in J(x)$ such that $f(y) = 0.$
	  The set of all such $y$ is denoted by $ x^{\perp_B} $, i.e., $x^{\perp_B} = \{ y \in \mathbb{X} : x \perp_B y \}. $ In \cite{S17}, the notion of Birkhoff-James orthogonality was dissected into two parts, namely, $x^+$ and $x^-$,  as follows: For  a given $x\in \mathbb{X},$ we say that $y \in x^+$ if $ \| x + \lambda y \| \geq \|x\|$ for all $\lambda\geq0$ and $y \in x^-$ if $ \| x + \lambda y \| \geq \|x\|,$ for all $\lambda\leq 0.$ For  various applications of Birkhoff-James orthogonality in studying the geometry of normed linear spaces, one may go through the recent book \cite{Book24}.\\
	
    We recall that  $  T \in \mathbb{L}(\mathbb{X}, \mathbb{Y}) $ is said to preserve  Birkhoff-James orthogonality if for  $ x, y \in \mathbb{X},$ $ x \perp_B y \Rightarrow Tx \perp_B Ty. $ Quite naturally, for the local version of this property, The operator $T$ is said to preserve Birkhoff-James orthogonality at $ x \in \mathbb{X} $ if for all $y \in \mathbb{X},$ $ x \perp_B y \Rightarrow Tx \perp_B Ty. $ Several recent studies \cite{S20,S18,SRT21} have established the importance of the  local preservation of Birkhoff-James orthogonality in understanding the geometric properties of the underlying  normed linear spaces. Recently in \cite{SMP24}, it was proved that  if a bounded linear operator $ T \in \mathbb{L}(\mathbb{X}, \mathbb{Y}) $ preserves Birkhoff-James orthogonality  at a $k$-smooth point $ x \in \mathbb{X} $ then $ Tx $ must be $(k+r)$-smooth in $ \mathbb{Y}, $ for some $ r \geq 0. $ One of our main aims is to improve this result. For this purpose, let us introduce the following two definitions.
    
		\begin{definition}\label{local-y}
		Let $ \mathbb{X}, \mathbb{Y} $ be  normed linear spaces and let $ x \in \mathbb{X}$ be non-zero. An operator $T\in\mathbb{L}(\mathbb{X}, \mathbb{Y})$ is said to preserve Birkhoff-James orthogonality at $x$ in the direction $y\in x^{\perp_B}$ if $Tx\perp_B Ty.$
	\end{definition}
	
	\begin{definition}\label{local-Z}
		Let $ \mathbb{X}, \mathbb{Y} $ be  normed linear spaces and let $ x\in\mathbb{X} $ be non-zero. Let $ \mathbb{Z} $ be a subspace of $ \mathbb{X} $ such that $\mathbb{Z} \subset x^{\perp_B}.$ An operator $T\in\mathbb{L}(\mathbb{X},\mathbb{Y})$ is said to preserve Birkhoff-James orthogonality at $x$ with respect to $\mathbb{Z}$ if $Tx\perp_B Tz,$ for all $z\in\mathbb{Z}.$ 
	\end{definition}
	
	It is worth noting that the local preservation of orthogonality at an element  in a particular direction, as mentioned in Definition \ref{local-y}, is not always guaranteed. As for example, consider $T \in \mathbb{L}(\ell_\infty^2) $ given by
	\[ T(1,1)=(0,1),~~T(-1,1)=(-1,0).\]
	Then for the element $ x = ( 2,1), $ there exists no $ y \in \ell_\infty^2 $ such that $ x \perp_B y $ and $ Tx \perp_B Ty. $  On the other hand, for $ x = (1,1) \in \ell_\infty^2, $ there exists $ y = (-1,1) \in \ell_\infty^2 $ such that $ x \perp_B y $ and $ Tx \perp_B Ty.$  However, in Proposition \ref{prop} of the present article, we will see that such an example is essentially a two-dimensional phenomenon.\\

	Clearly, Definition \ref{local-Z} generalizes the local preservation of orthogonality at a point $ x $ only when $ x $ is \textit{not} a smooth point. In that case, there exists a family of hyperspaces $H_{\alpha} $  such that $ x \perp_B H_\alpha $ and it may so happen  that $ x \perp_B H_{\alpha_0} \Rightarrow Tx \perp_B  T(H_{\alpha_0} ) $ for some fixed hyperspace $H_{\alpha_0} $ but the same is not true for any hyperspace $H_\alpha$ with $ \alpha \neq \alpha_0.$  To see this, it suffices to consider the same operator $ T \in \mathbb{L}(\ell_\infty^2), $ as given above, and to observe that $ (1,1) \perp_B (-1,1) $ and $ T(1,1) = (0,1) \perp_B (-1,0) = T(-1,1), $ whereas  $ (1,1) \perp_B (0,1) $ but $T(1,1)=(0,1) \not\perp_B (-\frac{1}{2},\frac{1}{2}) = T(0,1)$.\\ 
	
	 In this paper, we mainly study the local preservation of  Birkhoff-James orthogonality  at a point $x$ with respect to some subspace contained in $x^{\perp_B}.$ First we obtain a complete  characterization of the local preservation of  Birkhoff-James orthogonality  at a point. Next, we characterize the local preservation of  Birkhoff-James orthogonality  at a point in a particular  direction. Then we study  local preservation of Birkhoff-James orthogonality   with respect to the kernels of the support functionals. We also explore the connection between the order of smoothness of a point  and the directional preservation of Birkhoff-James orthogonality with respect to the kernels of  some support functionals. As an application of the present study, we obtain a refinement of the Blanco-Koldobsky-Turn\v{s}ek Theorem on some polyhedral  normed linear spaces.\\

	We end this section with the following results which find extensive use throughout this article.
	
	\begin{theorem}\cite[Th. 1.1, pp 170]{S70}\label{theoSinger}
		Let $ \mathbb{X} $ be a  normed linear space. Then for $ x,y\in\mathbb{X},$ $x\perp_By$ if and only if there exist $\phi_1,\phi_2\in Ext~J(x)$ and $\alpha\in[0,1]$ such that $(1-\alpha)\phi_1(y)+\alpha\phi_2(y)=0.$ 
	\end{theorem}
	
	\begin{theorem}\cite[Th. 2.4]{SMP20}\label{prop+ func}
		Let $ \mathbb{X} $ be a  normed linear space and let $x,y\in \mathbb{X}.$ Then the following are
		true:
		\begin{itemize}
			\item[(i)] $y\in x^+$ if and only if there exists an $f\in J(x)$ such that $f(y)\geq0.$
			\item[(ii)] $y\in x^-$ if and only if there exists an $f\in J(x)$ such that $f(y)\leq0.$
		\end{itemize}
	\end{theorem}
	
	\section{Main Results}
	
	We begin with the following simple observation that ensures the existence of subspaces on which an operator preserves Birkhoff-James orthogonality.

		\begin{prop}\label{prop}
		Let $\mathbb{X}$  and $\mathbb{Y} $ be   normed linear spaces and let $dim~\mathbb{X}> 2.$ Suppose that $x \in \mathbb{X}$ is non-zero. Then for any operator $T\in\mathbb{L}(\mathbb{X}, \mathbb{Y}),$ there exists a subspace $\mathbb{V}\subset x^{\perp_B}$ of codimension  $2$ in $ \mathbb{X} $ such that $T$ preserves Birkhoff-James orthogonality at $x$ with respect to  $\mathbb{V}.$
	\end{prop}
	\begin{proof}
	 If $Tx=0$ then the result holds trivially. Suppose $Tx\neq0.$ Let $f\in J(x)$ and $g\in J(Tx).$ Clearly, $\mathbb{Y}=\ker g \oplus span\{Tx\}.$  First we show that there exists $\mathbb{V}(\neq\{0\})\subset \ker f $ such that $T(\mathbb{V})\subset \ker g.$ Suppose on the contrary that for each $z(\neq 0)\in \ker f,$ $Tz \notin  \ker g.$ Since  $dim~\mathbb{X}> 2,$ there exist two linearly independent elements $u_1,u_2\in \ker f$ such that 
		\[Tu_1=h_1+\alpha_1Tx\text{ and } Tu_2=h_2+\alpha_2Tx,\text{ for some } h_1,h_2\in \ker g\text{ and }\alpha_1,\alpha_2\in \mathbb{R}\setminus\{0\}.\]
		This implies that $\alpha_2 u_1-\alpha_1 u_2(\neq 0)\in \ker f$ and $T(\alpha_2 u_1-\alpha_1 u_2)=\alpha_2h_1-\alpha_1h_2\in \ker g,$ a contradiction. Thus, there exists a subspace $\mathbb{V}(\neq\{0\})\subset \ker f $ such that $T(\mathbb{V})\subset \ker g.$\\
		Next, we claim that there exists a subspace $\mathbb{V}\subset \ker f$  of codimension 1 in $\ker f$ such that $T(\mathbb{V})\subset \ker g.$ Contrary to our claim suppose that $\mathbb{V}$ is a maximal (with respect to set inclusion) subspace of $\ker f$ such that $T(\mathbb{V})\subset \ker g,$ where $codim~\mathbb{V}\geq 2$ in $\ker f.$ Then there exist two linearly independent elements $v_1,v_2\in \mathbb \ker f \setminus \mathbb{V}$ such that $span\{v_1,v_2\}\cap \mathbb{V} =\{0\}$ and $Tv_1,Tv_2\notin \ker g.$ Then	
		\[Tv_1=k_1+\beta_1Tx\text{ and } Tv_2=k_2+\beta_2Tx,\text{ for some } k_1,k_2\in \ker g\text{ and }\beta_1,\beta_2\in \mathbb{R}\setminus\{0\}.\] 
		This implies that $\beta_2 v_1-\beta_1 v_2(\neq 0)\in \ker f\setminus\mathbb{V}$ and $T(\beta_2 v_1-\beta_1 v_2)=\beta_2k_1-\beta_1k_2\in \ker g.$ Let $\mathbb{V}_1=\mathbb{V}\oplus span\{\beta_2 v_1-\beta_1 v_2\}.$ Then clearly, $\mathbb{V}_1\subset \ker f $ and $T(\mathbb{V}_1)\subset \ker g,$ which  contradicts  the maximality of $\mathbb{V}.$ This establishes our claim that there exists a subspace $\mathbb{V}\subset \ker f$  of codimension 1 in $\ker f$ such that $T(\mathbb{V})\subset \ker g. $ Therefore, there exists a subspace $\mathbb{V}\subset x^{\perp_B}$ of codimension  $2$ in $ \mathbb{X} $ such that $T$ preserves Birkhoff-James orthogonality at $x$ with respect to  $\mathbb{V}.$ This completes the proof.

	\end{proof}
	
	Our next goal is to obtain a complete characterization of the local preservation of Birkhoff-James orthogonality. To this end, we introduce the notion of associated cones of a non-zero element $x\in \mathbb{X},$ corresponding to the extreme support  functionals at $ x. $  We recall that a subset $K$ of $\mathbb{X}$ is said to be a normal cone if 
		\[\text{(i)}~K+K\subset K,~\text{(ii)}~\alpha K\subset K, \text{ for all } \alpha \geq 0,~\text{(iii)}~K\cap(-K)=\{0\}.\]
		
\begin{definition}
	Let $\mathbb{X}$  be a  normed linear space and let $x \in \mathbb{X}$ be non-zero. Suppose that $Ext~J(x)=\{f_i:i\in \Lambda\},$ where $\Lambda$ is an index set. For $i,j\in\Lambda, $  the set $V_{ij}=\{y\in\mathbb{X}:f_i(y)\geq0\text{ and }f_j(y)\leq 0\}$ is said to be the associated cone of $x$ corresponding to $f_i\text{ and } f_j.$ We denote the set of all associated cones of $x$ by $\mathcal{V}(x).$
\end{definition}

We state some obvious but useful properties of the set of associated cones, which will be used later on in this article. The proofs are omitted as they can be completed using elementary arguments.

\begin{prop}
Let $\mathbb{X}$ be a  normed linear space and let $x\in \mathbb{X}$ be non-zero. Then the following results hold true:
	\begin{itemize}
		\item[(i)]  Associated cones of $x$ are closed and convex.
		\item[(ii)] Let $f_i,f_j\in Ext~J(x)$ for $i,j\in\Lambda.$ Then $V_{ij}=-V_{ji}.$
		\item[(iii)]  $x$ is smooth if and only if the set  $\mathcal{V}(x)$ is singleton. 
		\item[(iv)] If $x$ is a $k$-smooth point of $\mathbb{X}$, where $k\geq 2,$ then the set $\mathcal{V}(x) $ contains at least $k^2$ elements,  
		\item[(v)] Let $\mathbb{X}=\ell_{\infty}^n$ and $k\geq 2.$ Then $x$ is a $k$-smooth point of $\mathbb{X}$ if and only if there exist exactly $k^2$ elements in  $\mathcal{V}(x).$
		\item[(vi)] Let $\mathbb{X}=\ell_{1}^n$ and $k\geq 2.$ Then $x$ is a $k$-smooth point of $\mathbb{X}$ if and only if there exist exactly $4^{k-1}$ elements in  $\mathcal{V}(x).$
	\end{itemize}
\end{prop} 

We also require the following lemma for our purpose.

\begin{lemma}\label{Lem-asso-cone}
	Let $\mathbb{X}$  be a  normed linear space and let $x \in \mathbb{X}$ be non-zero. Then $x^{\perp_B}=\bigcup\limits_{V\in \mathcal{V}(x)}V.$ 
\end{lemma}

\begin{proof}
	Let $y\in \bigcup\limits_{V\in \mathcal{V}(x)}V .$ Then $y\in V$ for some  $V\in\mathcal{V}(x).$ Suppose $V$ is the associated cone of $x$ corresponding to $f_1\text{ and } f_2,$ for some $f_1,f_2\in Ext~(J(x)).$ Hence $f_1(y)\geq0\text{ and }f_2(y)\leq 0$ and so there exists $\alpha \in [0,1] $ such that $(1-\alpha)f_1(y)+\alpha f_2(y)=0.$ This implies $x\perp_By$ and so $y\in x^{\perp_B}.$ Hence$\bigcup\limits_{V\in \mathcal{V}(x)}V\subset  x^{\perp_B}.$\\
	Next, let $z\in  x^{\perp_B}.$ From Theorem \ref{theoSinger}, it follows that there exist $\phi_1,\phi_2\in Ext~J(x)$ such that $(1-\alpha)\phi_1(z)+\alpha\phi_2(z)=0,$ for some $\alpha \in [0,1].$ Hence $\phi_1(z)\phi_2(z)\leq 0$ and so $z$ belongs to the associated cone corresponding to $\phi_1\text{ and } \phi_2$ $(\text{ or, }\phi_2\text{ and } \phi_1).$ Thus, $z\in\bigcup\limits_{V\in \mathcal{V}(x)}V.$ Hence $x^{\perp_B}\subset\bigcup\limits_{V\in \mathcal{V}(x)}V.$  Therefore, $x^{\perp_B}=\bigcup\limits_{V\in \mathcal{V}(x)}V.$ This completes the proof of the lemma.
\end{proof}

We are now ready to present a complete characterization of the local preservation of Birkhoff-James orthogonality.

\begin{theorem}\label{theocharBJ}
	Let $\mathbb{X}$ and $\mathbb{Y}$ be  normed linear spaces and let $x \in\mathbb{X}$ be non-zero. Then $T\in\mathbb{L}(\mathbb{X},\mathbb{Y})$  preserves Birkhoff-James orthogonality at $x$ if and only if for each associated cone $V\in\mathcal{V}(x),$ there exists a convex set $S\subset (Tx)^{\perp_B}$ such that $T(V)\subset S.$ 
\end{theorem}
\begin{proof}
	We first prove the necessary part. Let $T\in\mathbb{L}(\mathbb{X},\mathbb{Y})$ be such that $T$  preserves Birkhoff-James orthogonality at $x.$ Then $T(x^{\perp_B})\subset (Tx)^{\perp_B}.$ Let $\mathcal{V}(x)$ be the set of all associated cones of $x.$ From  Lemma \ref{Lem-asso-cone}, it follows that $x^{\perp_B}=\bigcup\limits_{V\in \mathcal{V}(x)}V.$ So for each $V\in \mathcal{V}(x),$ $T(V)\subset (Tx)^{\perp_B}.$ Since $V$ is convex, it follows that $T(V)$ is a convex subset of $ (Tx)^{\perp_B}$ and this completes the proof of the necessary part.\\
	Next, we prove the sufficient part. Suppose that for each associated cone $V\in\mathcal{V}(x),$ there exists a convex set $S\subset (Tx)^{\perp_B}$ such that $T(V)\subset S\subset(Tx)^{\perp_B}.$ From  Lemma \ref{Lem-asso-cone}, it follows that $x^{\perp_B}=\bigcup\limits_{V\in \mathcal{V}(x)}V$ and so $T(x^{\perp_B})=T(\bigcup\limits_{V\in \mathcal{V}(x)}V)\subset (Tx)^{\perp_B}.$  This completes the proof.
\end{proof}

For an operator on a dual  normed linear space, we have the following refinement of Theorem \ref{theocharBJ}.

\begin{theorem}
	Let $\mathbb{X}$ and $\mathbb{Y}$ be   normed linear spaces and let $x^* \in\mathbb{X}^*$ be non-zero. Then $T\in\mathbb{L}(\mathbb{X}^*,\mathbb{Y})$  preserves Birkhoff-James orthogonality at $x^*$ if and only if for each associated cone $V\in\mathcal{V}(x^*),$ there exists a convex set $S\subset (Tx^*)^{\perp_B}$ such that $ T (Ext~(V\cap B_{\mathbb{X}^*})) \subset S.$
\end{theorem}

\begin{proof}
	Since the necessary part follows directly from Theorem \ref{theocharBJ}, we only prove the sufficient part. Suppose that for each associated cone $V\in\mathcal{V}(x^*),$ there exists a convex set $S\subset (Tx^*)^{\perp_B}$ such that  $T(Ext~(V\cap B_{\mathbb{X}^*}))\subset S.$
	Since $ B_{\mathbb{X}^*}$ is a weak*-compact convex subset of $\mathbb{X}^*$ and $V$ is a closed and convex subset of $\mathbb{X}^*,$ it follows that   $V\cap B_{\mathbb{X}^*}$ is a weak*-compact convex subset of $\mathbb{X}^*.$ So by the Krein-Milman Theorem, it follows that $V\cap B_{\mathbb{X}^*}=\overline{co(Ext~(V\cap B_{\mathbb{X}^*}))}^{w^*}.$ Now, by using the convexity of $ S $ and the continuity and the linearity of $ T, $ we obtain that
	\begin{eqnarray*}
		T(V\cap B_{\mathbb{X}^*})&=&T(\overline{co(Ext~(V\cap B_{\mathbb{X}^*}))}^{w^*})\\
		&\subset&\overline{T(co(Ext~(V\cap B_{\mathbb{X}^*})))} \\
		&=&	\overline{co(T(Ext~(V\cap B_{\mathbb{X}^*})))}\\
		& \subset & \overline{S}.
	\end{eqnarray*}
	Note that $ (Tx^*)^{\perp_B} $  is a closed set and $ S \subset  (Tx^*)^{\perp_B}. $ So  $ T(V\cap B_{\mathbb{X}^*}) \subset  (Tx^*)^{\perp_B} .$ 
	Now, \[\bigcup\limits_{V\in\mathcal{V}(x^*)}(V\cap B_{\mathbb{X}^*})=(\bigcup\limits_{V\in\mathcal{V}(x^*)}V)\cap B_{\mathbb{X}^*}=(x^*)^{\perp_B}\cap B_{\mathbb{X}^*},\]
	which shows that $T((x^*)^{\perp_B}\cap B_{\mathbb{X}^*})\subset (Tx^*)^{\perp_B}.$ Since Birkhoff-James orthogonality is homogeneous, it follows that $T((x^*)^{\perp_B})\subset(Tx^*)^{\perp_B}.$ This completes the proof.
\end{proof}

 We now turn our attention towards characterizing the local preservation of Birkhoff-James orthogonality at a point \textit{in a particular direction}. We require the following two lemmas for our purpose.
 
\begin{lemma}\label{lemmax+-}
	Let $\mathbb{X}$ be a  normed linear space and let $x \in \mathbb{X}$ be non-zero. Then
	\begin{itemize}
		\item[(i)] $x^+\setminus x^-=\big(\bigcup\limits_{f\in J(x)}\{z:f(z)>0\}\big)\cap (x^{\perp_B})^c.$
		\item[(ii)] $x^-\setminus x^+=\big(\bigcup\limits_{f\in J(x)}\{z:f(z)<0\}\big)\cap (x^{\perp_B})^c.$
	\end{itemize}
\end{lemma}

\begin{proof}
	We only prove (i) and note that (ii) can be proved using similar arguments.\\
	Let $y \in x^+\setminus x^-.$ Then clearly, $y\in(x^{\perp_B})^c.$ Now, $y \in x^+$ implies that there exists $f\in J(x)$ such that $f(y)\geq 0.$ Since $y \notin x^-,$ it follows that $f(y)>0$ and so 
	$y\in\big(\bigcup\limits_{f\in J(x)}\{z:f(z)>0\}\big)\cap (x^{\perp_B})^c.$ Thus,  $x^+\setminus x^-\subset\big(\bigcup\limits_{f\in J(x)}\{z:f(z)>0\}\big)\cap (x^{\perp_B})^c.$\\
	To prove the reverse containment, let $w\in\big(\bigcup\limits_{f\in J(x)}\{z:f(z)>0\}\big)\cap (x^{\perp_B})^c.$ So $x\not\perp_Bw$ and there exists $f\in J(x)$ such that $f(w)>0.$ Hence $w\in x^+.$ Since $x\not\perp_Bw,$ it follows that $w\notin x^-.$ Thus, $x^+\setminus x^-\supset\big(\bigcup\limits_{f\in J(x)}\{z:f(z)>0\}\big)\cap (x^{\perp_B})^c.$ Therefore, $x^+\setminus x^-=\big(\bigcup\limits_{f\in J(x)}\{z:f(z)>0\}\big)\cap (x^{\perp_B})^c.$
\end{proof}

\begin{lemma}\label{lemperpcon}
	Let $\mathbb{X}$ be a  normed linear space with $dim~ \mathbb{X}>2.$ Then for each $x\in \mathbb{X},$ the set $x^{\perp_B}\setminus \{0\}$ is connected.
\end{lemma}

\begin{proof} The result holds trivially if $ x = 0. $ Let $x\neq 0.$ Note that 
	\[x^{\perp_B}\setminus \{0\}=\bigcup\limits_{f\in J(x)}(\ker f \setminus \{0\}).\] 
	First we show that $\ker f \setminus \{0\}$ is path connected for all $f\in J(x).$ Since $dim ~\mathbb{X}>2,$ $dim~\ker f>1,$ for all $f\in J(x).$ Let $x,y\in \ker f \setminus \{0\}. $  We consider the following two cases separately:
	\medskip
	
	\noindent \textbf{Case 1:} Suppose that $x,y$ are linearly independent. Then $(1-\alpha)x+\alpha y\in \ker f \setminus \{0\},$ for all $\alpha \in [0,1]$ and so there exists a path in $ \ker f \setminus \{0\}$ between $x$ and $y.$
	\medskip
	
	\noindent \textbf{Case 2:} Suppose that $x,y$ are linearly dependent. Let $z\in \ker f \setminus \{0\} $ be such that $x,z$ are linearly independent.  Then there exists a path in $ \ker f \setminus \{0\}$ between $x$ and $z.$ Similarly, there exists a path in $ \ker f \setminus \{0\}$ between $y$ and $z.$ So there exists a path in $ \ker f \setminus \{0\}$ between $x$ and $y.$\\
	Thus, it follows from Case 1 and Case 2 that $\ker f \setminus \{0\}$ is path connected for all $f\in J(x).$\\
	Next, for any $f_1,f_2\in J(x),$ $dim~(\ker f_1 \cap \ker f_2)\geq 1.$ So $(\ker f_1 \setminus \{0\})\cap(\ker f_2 \setminus \{0\})\neq\emptyset,$ for all $f_1,f_2\in J(x).$ Therefore, $x^{\perp_B}\setminus \{0\}$ is path connected and hence connected. 
\end{proof}

We are now in a position to characterize the local preservation of Birkhoff-James orthogonality at a point in some direction.

\begin{theorem}\label{char dir}
	Let $\mathbb{X}$  and $\mathbb{Y} $ be  normed linear spaces with $dim~ \mathbb{X}>2$ and let $x \in \mathbb{X}$ be non-zero. Then $T\in\mathbb{L}(\mathbb{X}, \mathbb{Y})$ preserves Birkhoff-James orthogonality at $x$ in the direction $y\in x^{\perp_B}\setminus\{0\}$ if and only if there exist $u,v\in x^{\perp_B}\setminus\{0\}$ such that $Tu \in (Tx)^{+}$ and $Tv\in (Tx)^{-}.$ 
\end{theorem}

\begin{proof}
	We first prove the necessary part. Let $T\in\mathbb{L}(\mathbb{X}, \mathbb{Y})$ be such that $T$ preserves Birkhoff-James orthogonality at $x\in \mathbb{X}$ in  the direction $y$ for some $y\in x^{\perp_B}\setminus\{0\}.$ Then $Tx\perp_B Ty.$ This implies that $Ty \in (Tx)^{+}$ and $Ty\in (Tx)^{-},$ and completes the proof of the necessary part.\\
	We next prove the sufficient part. Let $T\in\mathbb{L}(\mathbb{X}, \mathbb{Y})$ and let  there exist $u,v\in x^{\perp_B}\setminus\{0\}$ be such that $Tu \in (Tx)^{+}$ and $Tv\in (Tx)^{-}.$
	Suppose on the contrary that there does not exist  $y\in x^{\perp_B}$ such that $Tx\perp_B Ty.$ Since $dim~\mathbb{X} >2, $ it follows from Lemma \ref{lemperpcon} that $x^{\perp_B}\setminus \{0\}$ is connected.
	Consider the following two sets 
	\[	A=\{y\in x^{\perp_B}\setminus\{0\}:Ty\in (Tx)^+\setminus (Tx)^-\},~B=\{y\in x^{\perp_B}\setminus\{0\}:Ty\in (Tx)^-\setminus (Tx)^+\}.\]
	Clearly, $x^{\perp_B}\setminus \{0\}=A\cup B$ and $A\cap B=\emptyset.$ From the hypothesis, it follows that $u\in A$ and $v\in B.$ Hence $A,B$ are nonempty. Using Lemma \ref{lemmax+-}, we get that 
	\[(Tx)^+\setminus (Tx)^-=\Big(\bigcup\limits_{f\in J(Tx)}\{z:f(z)>0\}\Big)\cap \big((Tx)^{\perp_B}\big)^c.\]
	Clearly, $(Tx)^+\setminus (Tx)^-$ is an open set in $\mathbb{Y}.$ Since $T$ is continuous, the set $\{y\in \mathbb{X}:Ty\in (Tx)^+\setminus (Tx)^-\}$ is open in $\mathbb{X}$ and so $\{y\in \mathbb{X}:Ty\in (Tx)^+\setminus (Tx)^-\}\cap (x^{\perp_B}\setminus \{0\})$ is open in $x^{\perp_B}\setminus \{0\}.$ Thus, $A$ is open in $x^{\perp_B}\setminus \{0\}.$ Similarly, $B$ is open in $x^{\perp_B}\setminus \{0\}.$ This contradicts the fact that  $x^{\perp_B}\setminus \{0\}$ is connected. Therefore, there exists  $y\in x^{\perp_B}\setminus\{0\}$ such that $Tx\perp_B Ty.$ This completes the proof.
\end{proof}

The following easy example shows that the above theorem may not be true in two-dimensional  normed linear spaces.

\begin{example}
	Let $\mathbb{X} = \ell_{\infty}^{2}.$ Consider the operator $ T\in\mathbb{L}(\mathbb{X})$ given by $T(x,y)=( 3x-y,4x-2y).$ Then it is easy to verify that $T$ does not preserve Birkhof-James orthogonality at $(1,1)$ but there exist $(1,0),(0,1)\in (1,1)^{\perp_B} $ such that $T(1,0) \in (T(1,1))^{+}$ and $T(0,1)\in (T(1,1))^{-}.$
\end{example}

In the two-dimensional case, we have the following characterization of the local preservation of Birkhoff-James orthogonality at a point in some direction. 
\begin{theorem}
	Let $\mathbb{X}$  be a two-dimensional  normed linear space and let $\mathbb{Y} $ be any  normed linear space. Let $x \in \mathbb{X}$ be non-zero. Then $T\in\mathbb{L}(\mathbb{X}, \mathbb{Y})$ preserves Birkhoff-James orthogonality at $x\in \mathbb{X}$ in the direction $y\in x^{\perp_B}\setminus\{0\}$ if and only if there exist $u,v\in x^{\perp_B}\setminus\{0\}$ such that $Tu \in (Tx)^{+}$ and $Tv\in (Tx)^{-}$ and $(1-\alpha)u+\alpha v\in x^{\perp_B}\setminus\{0\},$ for all $\alpha\in[0,1].$
\end{theorem}

\begin{proof}
	As the necessary part follows trivially, we only prove the sufficient part. Let $u,v\in x^{\perp_B}$ be such that $Tu \in (Tx)^{+}$ and $Tv\in (Tx)^{-}$ and $(1-\alpha)u+\alpha v\in x^{\perp_B}\setminus\{0\},$ for all $\alpha\in[0,1].$ We consider the following two cases separately.
	\medskip
	
	\noindent \textbf{Case 1:} Suppose that $x$ is not a smooth point and let $f_1,f_2\in Ext~J(x)$ be linearly independent. Then $ x^{\perp_B}\setminus\{0\}=D\cup(-D),$ where $D=\{y:f_1(y)\geq 0 \text{ and }f_2(y)\leq 0\}\setminus\{0\}.$ Clearly, $ x^{\perp_B}\setminus\{0\}$ is not connected and $D$ is connected.
	Since $(1-\alpha)u+\alpha v\in x^{\perp_B}\setminus\{0\},$ for all $\alpha\in[0,1],$ it follows that either $u,v\in D$ or $u,v\in -D.$ Without any loss of generality, let us assume that $u,v\in D.$ Suppose on the contrary that there does not exist  $z\in x^{\perp_B}$ such that $Tx\perp_B Tz.$ 
	Consider the following two sets 
	\[	D_1=\{y\in x^{\perp_B}:Ty\in (Tx)^+\setminus (Tx)^-\}\cap D,~D_2=\{y\in x^{\perp_B}:Ty\in (Tx)^-\setminus (Tx)^+\}\cap D.\]
	Then $D=D_1\cup D_2$ and $D_1\cap D_2=\emptyset.$ From the hypothesis, it follows that $u\in D_1$ and $v\in D_2.$ In particular, $D_1,D_2$ are nonempty. It follows from Lemma \ref{lemmax+-} that 
	\[(Tx)^+\setminus (Tx)^-=\Big(\bigcup\limits_{f\in J(Tx)}\{z:f(z)>0\}\Big)\cap \big((Tx)^{\perp_B}\big)^c.\]
	Clearly, $(Tx)^+\setminus (Tx)^-$ is an open set in $\mathbb{Y}.$ Since $T$ is continuous, the set $\{y\in \mathbb{X}:Ty\in (Tx)^+\setminus (Tx)^-\}$ is open in $\mathbb{X}$ and so $\{y\in \mathbb{X}:Ty\in (Tx)^+\setminus (Tx)^-\}\cap D$ is open in $D,$ i.e., $D_1$ is open in $D.$ Similarly, $D_2$ is open in $D.$ Thus, $D_1$  and $ D_2$ form a separation of $D,$ which contradicts the fact that  $D$ is connected.  Therefore, there exists  $y\in x^{\perp_B}\setminus\{0\}$ such that $Tx\perp_B Ty.$ 
	\medskip
	
	\noindent \textbf{Case 2:}  Suppose that $x$ is a smooth point. In this case, there exists a unique  $g \in J(x).$  Let $w\in \ker g.$ Then $ x^{\perp_B}\setminus\{0\}=D\cup-D,$ where $D=\{kw:k>0\}.$ Clearly, $ x^{\perp_B}\setminus\{0\}$ is not connected and $D$ is connected. 	Since $(1-\alpha)u+\alpha v\in x^{\perp_B}\setminus\{0\},$ for all $\alpha\in[0,1],$ it follows that either $u,v\in D$ or $u,v\in -D.$ Without any loss of generality, let us assume that $u,v\in D.$ Then $v=k_1u$ for some $k_1>0.$ Since $Tu \in (Tx)^{+}$, it follows that $k_1Tu \in (Tx)^{+}$ and so $Tv \in (Tx)^{+}.$ Hence $Tv\in (Tx)^{\perp_B},$  completing the proof of the theorem.
\end{proof}

We next study the preservation of Birkhoff-James orthogonality at a point with respect to the kernel of a support functional at that point and obtain some interesting consequences. First we prove the following lemma.

	\begin{lemma}\label{lemma-co-to-co}
	Let $\mathbb{X}$ and $\mathbb{Y}$ be  normed linear spaces. Let $f_1,f_2\in \mathbb{X}^*$ be linearly independent and let $g_1,g_2\in \mathbb{Y}^*.$ Suppose that $F=co(\{f_1,f_2\}),$ $G_1=co(\{g_1,g_2\})$ and 	$G_2=co(\{g_1,-g_2\}).$ Consider the sets $A=\bigcup\limits_{f\in F} \ker f,$ $B_1=\bigcup\limits_{g\in G_1} \ker g$ and $B_2=\bigcup\limits_{g\in G_2} \ker g.$ If $T\in\mathbb{L}(\mathbb{X},\mathbb{Y})$ is such that  $T(\ker f_1)\subset \ker g_1$ and $T(\ker f_2)\subset \ker g_2,$ then  $T(A)\subset B_1$ or $T(A)\subset B_2.$
\end{lemma}

\begin{proof}
	If $g_1,g_2$ are linearly dependent, then the result holds trivially.
	Suppose that $g_1,g_2$ are linearly independent. Consider the set
	\begin{eqnarray*}
		K&=&\{y\in A:f_1(y)>0\text{ and }f_2(y)<0\}.
	\end{eqnarray*} 
	Clearly, $A=\overline K\cup {-\overline K}.$ 
	It is easy to observe that 
	\begin{eqnarray*}
	\text{(i)}~T(K)\subset B_1(or~B_2)&\implies&T(-K)\subset B_1(or~B_2),\\
	\text{(ii)}~T(K)\subset B_1(or~B_2)&\implies&T(\overline K)\subset B_1(or~B_2).
	\end{eqnarray*}
	Therefore, it suffices to show that $T(K)\subset B_1$ or $T(K)\subset B_2.$\\
	If there exists $u\in \ker f_1\setminus \ker f_2 $ such that $g_2(Tu)=0.$ Then $\mathbb{X}$ can be written as $\mathbb{X}=span\{u\}\oplus \ker f_2.$ Therefore, for each $x\in\mathbb{X},$  $x=\alpha u+h$ for some $\alpha \in \mathbb{R}$ and $h\in \ker f_2.$ Now, $g_2(Tx)=g_2(T(\alpha u+h))=0,$ which implies that $T(\mathbb{X})\subset \ker g_2$ and so $T(K)\subset B_1,B_2.$ 
	Similarly, if there exists $v\in \ker f_2\setminus \ker f_1$ such that $g_1(Tv)=0,$ then  $T(\mathbb{X})\subset \ker g_1$ and so $T(K)\subset B_1,B_2.$ \\
	Next, suppose that for each $u\in \ker f_1\setminus \ker f_2 ,$ $g_2(Tu)\not=0$ and for each  $v\in \ker f_2\setminus \ker f_1,$ $g_1(Tv)\not=0.$
	Let $y\in K.$ Then $f_1(y)>0\text{ and }f_2(y)<0.$ Consider $z\in \ker f_1$ such that $f_2(z)<f_2(y).$ Then for $t=\frac{-f_2(y)}{f_2(z)-f_2(y)} <0,$ we have $tz+(1-t)y\in \ker f_2.$ Let $w=tz+(1-t)y$ and so $y=\frac{-t}{1-t}z+\frac{1}{1-t}w.$ Therefore, for each $y\in K,$ there exist $z_y\in \ker f_1$ and $w_y\in \ker f_2$ such that 
	\[f_2(z_y)<0, f_1(w_y)>0 \text{ and } y=(1-\beta)z_y+\beta w_y \text{ for some }\beta \in (0,1).\]
	Then $g_1(Ty)g_2(Ty)=\beta(1-\beta) g_1(Tw_y)g_2(Tz_y).$\\
	We claim that for all $u\in \ker f_1 $ with $f_2(u)<0,$ $g_2(Tu)$ are of same sign. Contrary to our claim suppose that there exist $u_1,u_2\in \ker f_1$ with $f_2(u_1)<0$ and $f_2(u_2)<0$ such that 
	\[g_2(Tu_1)g_2(Tu_2)<0.\]
	Then there exists $0<\lambda<1$ such that 
	\[	(1-\lambda)g_2(Tu_1)+\lambda g_2(Tu_2)=0,~i.e.,~g_2(T((1-\lambda)u_1+\lambda u_2))=0.\]
	So, $T((1-\lambda)u_1+\lambda u_2)\in \ker g_2.$ It is easy to see that $(1-\lambda)u_1+\lambda u_2\notin \ker f_2.$ Let $u_3=(1-\lambda)u_1+\lambda u_2.$  Then $u_3\in \ker f_1\setminus \ker f_2 $ such that $g_2(Tu_3)=0.$ This contradicts the hypothesis that $g_2(Tu)\neq0,$ for all $u\in \ker f_1 $ with $f_2(u)<0.$ Thus our claim is established.\\
	Proceeding similarly we can show that for all $v\in \ker f_2 $ with $f_1(v)>0,$ $g_1(Tv)$ are of same sign. Therefore, either of the following two conditions hold true:
	\begin{eqnarray*}
	\text{(iii)}~g_1(Ty)g_2(Ty)&\geq&0, \text{ for all } y\in K.\\
	\text{(iv)}~g_1(Ty)g_2(Ty)&\leq&0, \text{ for all } y\in K.		 
	\end{eqnarray*}
	Thus, $T(K)\subset B_1$ or $T(K)\subset B_2,$ which completes the proof.
\end{proof}

 As an immediate consequence of Lemma \ref{lemma-co-to-co}, the following corollary is now immediate. 

\begin{cor}\label{conetocone}
	Let $\mathbb{X}$ and $\mathbb{Y}$ be  normed linear spaces. Let $f_1,f_2\in \mathbb{X}^*$ be linearly independent and let $g_1,g_2\in \mathbb{Y}^*.$ Suppose that $K=\{y\in \mathbb{X}:f_1(y)\geq 0\text{ and }f_2(y) \leq 0\},$ $K_1=\{y\in \mathbb{Y}:g_1(y) \geq 0\text{ and }g_2(y) \leq 0\}$ and 	$K_2=\{y\in \mathbb{Y}:g_1(y) \geq 0\text{ and }g_2(y) \geq 0\}.$ If $T\in\mathbb{L}(\mathbb{X},\mathbb{Y})$ is such that  $T(\ker f_1)\subset \ker g_1$ and $T(\ker f_2)\subset \ker g_2,$ then $T( K)\subset \pm  K_1$ or $T( K)\subset \pm  K_2.$
\end{cor}

It is worth noting in this context that $ x^{\perp_B} = \bigcup\limits_{f\in J(x)} \ker f. $ Our next goal is to characterize the directional preservation of Birkhoff-James orthogonality at a point with respect to the kernel of a support functional. We require the following lemmas for our purpose. In the first lemma, we show that if $\mathbb{V} $ is a subspace of $\mathbb{X}$ contained in $  x^{\perp_B},$   then $\mathbb{V}$ is contained in the kernel of some $f\in J(x).$


\begin{lemma}\label{subspace}
	Let $\mathbb{X}$ be a   normed linear space and let $x\in\mathbb{X}$ be non-zero. Suppose that $\mathbb{V}$  is a subspace of $\mathbb{X}.$ Then $\mathbb{V}\subset x^{\perp_B}$ if and only if there exists  $f\in J(x)$ such that $\mathbb{V}\subset \ker f.$
\end{lemma}
\begin{proof}
	The sufficient part of the proposition follows trivially, and accordingly, we only prove the necessary part.
	Let $\mathbb{Y}=span\{x\}\oplus \mathbb{V}.$  Consider the linear functional $ f \in \mathbb{Y}^{*}$ given by 
	\[	f(\alpha x+y)=\alpha\|x\|,~\text{ for all } y\in\mathbb{V} \text{ and } \alpha\in \mathbb{R}.\]
	We claim that $\| f\|=1.$ Let $y\in\mathbb{Y}$ be arbitrary. Then $y=\alpha x+v,$ for some $\alpha\in \mathbb{R}$ and $v\in\mathbb{V}.$ Since $x\perp_B v,$ it follows that $\|y\|\geq|\alpha|\|x\|.$ This implies $|f (y)|=|\alpha|\|x\|\leq \|y\|.$ Therefore, it follows that $\| f\|\leq 1.$ As $ f(x)=\|x\|,$ it is clear that $\| f\|=1,$ proving our claim. By the Hahn-Banach Theorem,  we extend $f$ to a linear functional $\tilde f \in \mathbb{X}^{*}$  such that $\tilde f|_{\mathbb{V}}=f$ and $\|\tilde f\|=1.$ Thus, $\tilde f(x)=f(x)=\|x\|$ and so $\tilde f\in J(x).$ Also, $\tilde f(\mathbb{V})=\{0\}$ and so $\mathbb{V}\subset \ker \tilde f.$ This completes the proof.
\end{proof}

 Our next lemma deals with the case when the subspace $ \mathbb{V} $ is contained in the union of the kernels of all possible convex combinations of two functionals.
 
	\begin{lemma}\label{Lem subg}
	Let $\mathbb{V}$ be a subspace of a  normed linear space $\mathbb{X}.$ Suppose that $g_1,g_2\in \mathbb{X}^*$ are linearly independent and let $G=co(\{g_1,g_2\}).$ Then  $\mathbb{V}\subset \bigcup\limits_{g\in G} \ker g$ if and only if there exists $g\in G$ such that $\mathbb{V}\subset\ker g.$
\end{lemma}
\begin{proof}
	The sufficient part follows trivially, and therefore, we only prove the necessary part.
	Suppose on the contrary that there does not exist $g\in G$ such that $\mathbb{V}\subset\ker g.$ So there exist $u,v\in \mathbb{V}$ and $\alpha,\beta\in[0,1]$ with $ \alpha \neq \beta $ such that 
	\begin{eqnarray}\label{eq1}
		(1-\alpha)g_1(u)+\alpha g_2(u)=0,
	\end{eqnarray} 
	\begin{eqnarray}\label{eq2}
		(1-\alpha)g_1(v)+\alpha g_2(v)\neq0,
	\end{eqnarray} 
	\begin{eqnarray}\label{eq3}
		(1-\beta)g_1(u)+\beta g_2(u)\neq0,
	\end{eqnarray} 
	\begin{eqnarray}\label{eq4}
		(1-\beta)g_1(v)+\beta g_2(v)=0.
	\end{eqnarray} 
	We claim that $\alpha\neq1.$ Contrary to our claim suppose that $\alpha=1.$ So $\beta\neq1.$ From (\ref{eq1}) and (\ref{eq3}), we have $g_2(u)=0 $ and $g_1(u)\neq0.$ From (\ref{eq2}), we have $g_2(v)\neq0 .$ So 
	\[\begin{vmatrix}
		g_1(u) & g_1(v) \\
		g_2(u) & g_2(v) 
	\end{vmatrix} \neq 0.\]
		From this it is easy to see that there exist $r,s\in \mathbb{R}$ such that
	 \[g_1(ru+sv)=1\text{ and }g_2(ru+sv)=1.\]
	Since $ru+sv\in \mathbb{V},$ this is clearly a contradiction, thus establishing our claim. Using similar arguments, we have that $\beta\neq1.$ Then clearly, $g_2(u)\neq 0$ and $g_2(v)\neq 0.$ Next, from (\ref{eq1}) and (\ref{eq2}), we have 
	\[ \frac{g_1(u)}{g_2(u)} = -\frac{\alpha}{1-\alpha} \neq \frac{g_1(v)}{g_2(v)}. \]
	Therefore, 
	\[\begin{vmatrix}
		g_1(u) & g_1(v) \\
		g_2(u) & g_2(v) 
	\end{vmatrix} \neq 0.\]
Then there exist $r,s\in \mathbb{R}$ such that
\[g_1(ru+sv)=1\text{ and }g_2(ru+sv)=1.\]
	Since $ru+sv\in \mathbb{V},$ this is a contradiction. Therefore, there exists $g\in G $ such that $\mathbb{V}\subset\ker g.$
	
\end{proof}

The following remark further clarifies the importance of the above lemma.

\begin{remark}
In general, if a subspace is contained in a collection of hyperspaces then it is not necessarily true that the subspace is contained in one of the hyperspaces. For example, let $\mathbb{X}=\ell_{\infty}^3.$  Consider $f_1,f_2,f_3 \in \mathbb{X}^{*} $ given by \[f_1(x,y,z)=x,~f_2(x,y,z)=y,~f_3(x,y,z)=z.\] 
Let $F=co({f_1,-f_1,f_2})$ and let $ \mathbb{V} = \ker f_3. $ Then it is elementary to see that $ \mathbb{V} \subset \underset{f\in F}{\cup}\ker f$ but $\mathbb{V} \not\subset \ker f $ for any $ f \in F.$ 
\end{remark}

Now, we are in a position  to characterize the directional preservation of Birkhoff-James orthogonality at a point with respect to the kernel of a support functional.

\begin{theorem}\label{kerf-subset-kerg}
	Let $\mathbb{X}$ and $\mathbb{Y}$ be   normed linear spaces and let $x\in\mathbb{X}$ be non-zero. Let $T\in\mathbb{L}(\mathbb{X},\mathbb{Y})$ be such that $Tx\neq0.$ Then the following results hold:
		\begin{itemize}
		\item[(i)] Let $f\in J(x).$ Then $T$ preserves Birkhoff-James orthogonality at $x$ with respect to $\ker f$ if and only if there exists $g\in J(Tx)$ such that $T(\ker f)\subset \ker g.$	
       \item[(ii)] If $T$ preserves Birkhoff-James orthogonality at $x$ with respect to $\ker f_i,$ for all $1\leq i\leq n,$ then  there exist $g_1,g_2,\dots,g_n \in J(Tx)$ such that $T(\ker f_i)\subset \ker g_i,$ for all $1\leq i\leq n.$ Moreover, 
	for each $f\in co(\{f_1,f_2,\dots,f_n\}),$ there exists $g\in co(\{g_1,g_2,\dots,g_n\})$ such that $T(\ker f)\subset \ker g.$
	\end{itemize}
\end{theorem}
\begin{proof}
(i) As the sufficient part follows easily, we only prove the necessary part. Since $T$ preserves Birkhoff-James orthogonality at $x$ with respect to $\ker f$, it follows that $T(\ker f) \subset (Tx)^{\perp_B} . $ Clearly, $T(\ker f) $ is a subspace of $\mathbb{Y} $ and $ (Tx)^{\perp_B} = \bigcup\limits_{g\in J(Tx)} \ker g.$ Therefore, by Lemma \ref{subspace}, we get $g\in J(Tx)$ such that $T(\ker f)\subset \ker g.$ This completes the proof. \\

(ii)  The first part of (ii) follows easily from (i). We complete the proof of second part of (ii) in two steps. In Step 1, we show that for each $f\in co(\{f_1,f_2\}), $ there exists $g\in co(\{g_1,g_2\})$ such that $T(\ker f)\subset \ker g.$  In Step 2, we complete the proof by using the principle of induction.

\medskip

\noindent \textbf{Step 1:} If $f_1,f_2$ are linearly dependent then there is nothing to prove. Let $f_1,f_2$ be linearly independent. Suppose that $F_1=co(\{f_1,f_2\}),$ $F_1'=co(\{f_1,-f_2\}),$ $G_1=co(\{g_1,g_2\})$ and $G_1'=co(\{g_1,-g_2\}).$ Consider the sets $A_1=\bigcup\limits_{f\in F_1} \ker f,$ $A_2=\bigcup\limits_{f\in F_1'} \ker f,$ $B_1=\bigcup\limits_{g\in G_1} \ker g$ and $B_2=\bigcup\limits_{g\in G_1'} \ker g.$ Clearly, $\mathbb{X}=A_1\cup A_2.$ It follows from  Lemma \ref{lemma-co-to-co} that for $i=1,2,$
\[	\text{either }T(A_i)\subset B_1\text{ or } T(A_i)\subset B_2.\]
It is easy to observe that $Tx\notin B_1$ and $x\in A_2.$ This implies that $T(A_2)\subset B_2.$ \\
 Now, we claim that $T(\ker f_1)\setminus\ker g_2\neq\emptyset.$ Suppose on the contrary that $T(\ker f_1)\subset \ker g_2.$ Then there exists $u\in \ker f_1\setminus \ker f_2$ such that $g_2 (Tu)=0.$ Now, $\mathbb{X}=span\{u\}\oplus \ker f_2 $ implies that for each $z\in\mathbb{X},$ $z$ can be expressed as $z=\alpha u+h$ for some $\alpha \in \mathbb{R}$ and some $h\in \ker f_2.$ Hence $g_2(Tz)=g_2(T(\alpha u+h))=0.$ Thus, $T(\mathbb{X})\subset \ker g_2,$ which contradicts the fact that $Tx\notin \ker g_2.$
Next, we show that $T(A_1)\subset B_1.$ Suppose on the contrary that 
$T(A_1)\subset B_2.$ Then $T(\mathbb{X})\subset B_2.$ Clearly, $g_1(Tx)>0\text{ and }g_2(Tx)>0.$ Since $T(\ker f_1)\setminus\ker g_2\neq\emptyset,$ 
it follows that there exists $x_1\in \ker f_1$ such that $Tx_1\notin\ker g_2.$ So we can find a  $k\in\mathbb{R}$ such that $kg_2(Tx_1)>g_2(Tx).$ Let $x_2=x-kx_1.$ Then $g_1(Tx_2)>0$ and $g_2(Tx_2)<0.$ Hence $Tx_2\notin B_2.$ This contradicts the hypothesis that $T(\mathbb{X})\subset B_2,$ proving our claim. Hence $T(A_1)\subset B_1,$ i.e., $T\Big(\bigcup\limits_{f\in F_1} \ker f\Big)\subset \bigcup\limits_{g\in G_1} \ker g.$ From Lemma  \ref{Lem subg}, it follows that for each $f\in F_1 $ there exists $g\in G_1$ such that $T(\ker f)\subset \ker g.$\\

\medskip

\noindent \textbf{Step 2:}  Let $P(n)$ represent the statement ``for each $f\in co(\{f_1,f_2,\dots,f_n\}),$ there exists $g\in co(\{g_1,g_2,\dots,g_n\})$ such that $T(\ker f)\subset \ker g$''. From (i), it follows that $P(1)$ holds.  Assuming that $P(k-1) $ holds  we  next show that $P(k) $  holds. 
Consider the set $F_k=co(\{f_1,f_2,\dots,f_k\})$ and $G_k=co(\{g_1,g_2,\dots,g_k\}).$ 
Let $\phi\in F_k.$ Then $\phi=\sum_{i=1}^{k}c_if_i,$ where $c_i\geq0$ and $\sum_{i=1}^{k}c_i=1.$
If $c_j=0,$ for all $1\leq j\leq k-1,$ then $ T(\ker \phi)\subset  \ker g_{k}$  and so $P(k)$ holds. Suppose $c_l\neq0$ for some $1\leq l\leq k-1.$
Then $\phi$ can be written as  
\[ \phi=\Big(\sum_{j=1}^{k-1}c_j\Big)\phi_1+c_{k}f_{k}, \, \mbox{where} \,  \phi_1=\sum_{i=1}^{k-1}\frac{c_t}{\sum_{j=1}^{m}c_j} f_i.\]
Clearly, $\phi_1\in F_{k-1}$ and so there exists $\psi_1 \in G_{k-1}$ such that $T(\ker \phi_1)\subset \ker \psi_1.$ Now, it follows from Step 1 that for   $ \phi \in co(\{\phi_1,f_{k}\}),$ there exists $\psi \in co(\{\psi_1,g_{k}\})$ such that   $ T(\ker \phi)\subset \ker \psi.$ Thus, $P(k)$ holds.  Hence by the principle of mathematical induction, $P(n) $ holds and this completes the proof of (ii). 
\end{proof}

The following corollary is an  immediate  consequence of the previous theorem. 
\begin{cor}\label{corpre}
	Let $\mathbb{X}$ and $\mathbb{Y}$ be  normed linear spaces and let $x \in \mathbb{X}$ be non-zero. Let $f_1,f_2,\dots,f_n \in J(x).$ Then $T\in\mathbb{L}(\mathbb{X},\mathbb{Y})$ preserves Birkhoff-James orthogonality at $x$ with respect to $\ker f_i,$ for all $1\leq i\leq n$  if and only if $T$ preserves Birkhoff-James orthogonality at $x$ with respect to $\ker f,$ for all $f\in co(\{f_1,f_2,\dots,f_n\}).$
\end{cor}

The next theorem shows that extreme support functionals play a very special role in the local preservation of Birkhoff-James orthogonality. 

	\begin{theorem}\label{thchprwrt}
	Let $\mathbb{X}$ and $\mathbb{Y}$ be  normed linear spaces and let  $x \in \mathbb{X}$ be non-zero. Then $T\in\mathbb{L}(\mathbb{X},\mathbb{Y})$ preserves Birkhoff-James orthogonality at $x$ if and only if $T$ preserves Birkhoff-James orthogonality at $x$  with respect to $\ker f,$ for all $f\in Ext~J(x).$
\end{theorem}
\begin{proof}
	The necessary part of the theorem being obvious, we only prove the sufficient part.	Let $T\in \mathbb{L}(\mathbb{X},\mathbb{Y}) $ be such that $T$ preserves Birkhoff-James orthogonality at $x$ with respect to $\ker f,$ for all $f\in Ext~J(x).$ Let $y\in x^{\perp_B}$ be arbitrary. From Theorem \ref{theoSinger}, it follows that there exist $\phi_1,\phi_2\in Ext~J(x)$ such that $y\in \ker((1-\alpha)\phi_1+\alpha\phi_2),$ for some $\alpha \in [0,1].$  From the hypothesis, it follows that $T$ preserves Birkhoff-James orthogonality at $x$ with respect to $\ker \phi_1$ and $\ker \phi_2.$ Now, it follows from Corollary \ref{corpre} that $T$ preserves Birkhoff-James orthogonality at $x$ with respect to $\ker \phi,$ for all $\phi\in co(\{\phi_1,\phi_2\})$ and so $Ty\in(Tx)^{\perp_B}.$ Since $y\in x^{\perp_B}$ is chosen arbitrarily, $T(x^{\perp_B})\subset(Tx)^{\perp_B}.$ Therefore, $T$ preserves  Birkhoff-James orthogonality at $x.$ This completes the proof.
\end{proof}

The following  example  shows that  local preservation of Birkhoff-James orthogonality at $x$ with respect to $\ker f ,$ \textit{for all} $ f \in Ext~ J(x)$ is essential for the above theorem to hold true. Indeed, Theorem \ref{thchprwrt} need not be true even if the operator preserves Birkhoff-James orthogonality at $x$ with respect to  $\ker f_i,$ for all $1\leq i\leq k, $ where $f_1,f_2,\dots,f_k \in Ext~J(x)$ are linearly independent and $x$ is $k$-smooth.
\begin{example}
	Let $\mathbb{X}=\ell_{1}^{3}.$ Consider $T\in \mathbb{L}(\mathbb{X})$ defined by \[T(\alpha,\beta,\gamma)=(\alpha+\beta,-\beta,-\beta-\gamma), \text{ for all }(\alpha,\beta,\gamma)\in \mathbb{X}.\]
	Let $x=(1,0,0).$ Consider $f_1,f_2,f_3 \in S_{\mathbb{X}^*}$ given by
	\[	f_1(\alpha,\beta,\gamma)=\alpha+\beta+\gamma,~	f_2(\alpha,\beta,\gamma)=\alpha+\beta-\gamma,~
		f_3(\alpha,\beta,\gamma)=\alpha-\beta-\gamma.\]
	It is easy to observe that $ x$ is $3$-smooth and $f_1,f_2,f_3\in Ext~J(x),$ which are linearly independent. We show that $T$ preserves Birkhoff-James orthogonality at $x$ with respect to $\ker f_i$ for each $ i=1,2,3$ but $T$ does not not preserve Birkhoff-James orthogonality at $x.$ Clearly, $\ker f_1= span\{(1,-1,0),(0,1,-1)\}.$ Let $y\in \ker f_1$ be arbitrary. It is easy to see that $y=(a,b-a,-b)$ for some $a,b\in \mathbb{R}.$ Now, $Ty=(b,a-b,a)$ and therefore,  for any $\lambda\in\mathbb{R},$ \[\|Tx+\lambda Ty\|=|1+\lambda b|+|\lambda a-\lambda b|+|\lambda a|\geq |1+\lambda b+\lambda a-\lambda b-\lambda a|=1=\|Tx\|.\] Hence $Tx\perp_B Ty.$ Since $y\in \ker f_1$ is chosen arbitrarily, it follows that $T$ preserves  Birkhoff-James orthogonality at $x$ with respect to $\ker f_1.$ Similarly, we can show that $T$ preserves  Birkhoff-James orthogonality at $x$ with respect to  $\ker f_2$ and $\ker f_3.$ 
	But $T$ does not preserve Birkhoff-James orthogonality at $x$  as $x\perp_B (2,1,-1)$ but $Tx\not\perp_B T(2,1,-1).$
\end{example}

We next illustrate that for an operator on a two-dimensional  normed linear space, the set of all directions in which the operator preserves Birkhoff-James orthogonality at a given point has a nice geometric structure. 

\begin{theorem}\label{cone}
	Let $\mathbb{X}$ be a two-dimensional  normed linear space and let $x\in \mathbb{X}$ be non-smooth. Suppose that $T\in \mathbb{L}(\mathbb{X}).$ If  $V\in \mathcal{V}(x)$ is the associated cone of $x$ corresponding to two linearly independent $h_1,h_2\in Ext~J(x),$ then the set $D=\{y\in x^{\perp_B}:Tx\perp_B Ty\}\cap V$ is a normal cone.
\end{theorem}
\begin{proof}
	Let us suppose that $V$ is the associated cone of $x$ corresponding to two linearly independent $h_1,h_2\in Ext~J(x).$ Then $V=\{y\in x^{\perp_B}:h_1(y)\geq0\text{ and } h_2(y)\leq0\}.$ Consider the set $A=\{y\in x^{\perp_B}:Tx\perp_B Ty\}.$ Let $D=A\cap V.$\\
	We show that $D$ is a normal cone.  Let $y_1,y_2\in D.$ Then $h_1(y_1)\geq0, h_2(y_1)\leq0\text{ and }h_1(y_2)\geq0, h_2(y_2)\leq0.$ Since $y_1,y_2\in x^{\perp_B},$ there exist $f_1,f_2\in J(x)$ such that $f_1(y_1)=0 \text{ and }f_2(y_2)=0.$ It is easy to observe that $T $ preserves Birkhoff-James orthogonality at $x$ with respect to $\ker~f_1 \text{ and } \ker~f_2.$ Clearly,
	\[ f_1 = (1-\alpha)h_1+\alpha h_2,~ f_2 = (1-\beta)h_1+\beta h_2,~ \textit{for some}~ \alpha, \beta \in [0, 1]. \]
	We claim that $f_1(y_2)f_2(y_1)\leq0.$ If $\alpha=0$ or $\beta=0,$ then clearly, our claim is true. Suppose that $\alpha \neq 0\text{ and }\beta \neq 0.$ Then 
	\[	f_1(y_1)=0\implies(1-\alpha)h_1(y_1)+\alpha h_2(y_1)=0
		\implies h_2(y_1)=-\frac{1-\alpha}{\alpha}h_1(y_1).\]
Similarly,		
\[f_2(y_2)=0\implies(1-\beta)h_1(y_2)+\beta h_2(y_2)=0
		\implies h_2(y_2)=-\frac{1-\beta}{\beta}h_1(y_2).\]
	Therefore,
	\begin{eqnarray*}
		f_1(y_2)f_2(y_1)&=&((1-\alpha)h_1(y_2)+\alpha h_2(y_2))((1-\beta)h_1(y_1)+\beta h_2(y_1)) \\
		&=&\left((1-\alpha)h_1(y_2)-\frac{\alpha(1-\beta)}{\beta}h_1(y_2)\right)\left((1-\beta)h_1(y_1)-\frac{\beta(1-\alpha)}{\alpha}h_1(y_1)\right)\\
		&=&\left(\frac{\beta(1-\alpha)-\alpha(1-\beta)}{\beta}\right)\left(\frac{\alpha(1-\beta)-\beta(1-\alpha)}{\alpha}\right)h_1(y_1)h_1(y_2)\\
		&=&\frac{(\beta-\alpha)(\alpha-\beta)}{\alpha\beta}h_1(y_1)h_1(y_2)\\
		& \leq & 0,
	\end{eqnarray*}
	establishing our claim. Let $y_3=c_1y_1+c_2y_2,$ $c_1,c_2\geq0.$ Clearly, $y_3\in V_{12}.$ Now,
	$f_1(y_3)f_2(y_3)=c_1 c_2f_1(y_2)f_2(y_1)\leq0$ and so there exists $t\in[0,1]$ such that $(1-t)f_1(y_3)+tf_2(y_3)=0,$ i.e., $y_3\in \ker~((1-t)f_1+tf_2).$ It follows from Corollary \ref{corpre} that $T$ preserves Birkhoff-James orthogonality at $x$ with respect to $\ker~((1-t)f_1+tf_2)$ and so $Tx\perp_BTy_3.$ Thus, $y_3\in D $ and therefore, $ D $ is a normal cone. This completes the proof of the theorem.
\end{proof}

As an easy consequence of Theorem \ref{cone}, we get the following corollary which shows that the orthogonality preserving set by an operator on a two-dimensional  normed linear space at a non-zero point can be expressed as the union of two normal cones. 
\begin{cor}
	Let $\mathbb{X}$ be a two-dimensional  normed linear space and let $x \in \mathbb{X} $ be non-zero. Suppose that $T\in \mathbb{L}(\mathbb{X}).$ Then the set $A=\{y\in x^{\perp_B}:Tx\perp_B Ty\}$ can be expressed as $A=D\cup (-D),$ where $D$ is a normal cone.
\end{cor}

\begin{remark}
	Theorem \ref{cone}  fails to hold if the dimension of the space is strictly greater than $2.$  Consider  $\mathbb{X}=\ell_{\infty}^{3}$ and  $T\in \mathbb{L}(\mathbb{X})$ defined as \[T(\alpha,\beta,\gamma)=(\alpha-\beta+\gamma,\alpha,2\gamma-\beta), \text{ for all }(\alpha,\beta,\gamma)\in \mathbb{X}.\]
	Let $x=(1,1,1).$ Then $Tx=x.$ Consider $f_1,f_2 \in S_{\mathbb{X}^*}$ given by
		\[f_1(\alpha,\beta,\gamma)=\alpha,~f_2(\alpha,\beta,\gamma)=\beta.\]
	It is easy to observe that $f_1,f_2 \in Ext~J(x).$ Let $V_{12}$ be the associated cone corresponding to $f_1 \text{ and } f_2.$ Let $D=\{y\in x^{\perp_B}:Tx\perp_B Ty\}\cap V_{12}.$ Clearly, $V_{12}=\{(\alpha,\beta,\gamma)\in \mathbb{X}:\alpha\geq 0\text{ and }\beta\leq0\}.$ Now, $y_1=(1,0,0),y_2=(0,-1,0)\in V_{12}.$ It is easy to observe that $Tx\perp_BTy_1$ and $Tx\perp_BTy_2.$ So $y_1,y_2\in D.$ But $Tx\not\perp_BT(\frac{y_1+y_2}{2})$ and so $\frac{y_1+y_2}{2}\notin D.$ Thus, $D$ is not convex and in particular, $D$ is not a normal cone. 
\end{remark}

Our next aim is to explore the relation between the orders of smoothness of $x$ and $Tx,$ when the operator $T$ locally preserves Birkhoff-James orthogonality at $x$ with respect to $\ker f,$ for some $ f \in \mathbb{X}^{*}. $ It is worth mentioning that the central idea of Theorem $ 2.9 $ of \cite{SMP24} depends upon such a connection. First we recall the following observation from \cite{SMP24}. 

\begin{lemma}\cite[Th. 2.7]{SMP24}\label{lem cotoaf}
	Let $ \mathbb{X} $ and $ \mathbb{Y} $ be  normed linear spaces of dimensions $n$ and $m,$ respectively. Let $ \{f_1, f_2, ..., f_k \}$ be a set of $ k $ linearly independent functionals in $\mathbb{X}^*$ and let $ \{g_1, g_2, ..., g_p \}$ be a set of $p~( <k)$ linearly independent functionals in $\mathbb{Y}^*.$ Let $F= co(\{f_1, f_2, ..., f_k \})$ and let
	$$G=\Big \{\sum_{i=1}^{p}\beta_ig_i :\sum_{i=1}^{p}\beta_i=1, \beta_i\in \mathbb{R} \,\, \forall\, i=1,2,\ldots,p \Big \}.$$
	Let $A=\bigcup\limits_{f\in F}\ker(f)$ and  $B=\bigcup\limits_{g\in G}\ker(g).$ If $T\in\mathbb{L}(\mathbb{X},\mathbb{Y})$ is such that $T(A)\subset B,$ then $T(\mathbb{X})\subset B .$
\end{lemma}

In \cite[Th. 2.9]{SMP24}, it was shown that the local preservation of Birkhoff-James orthogonality at a point by a linear operator also preserves the order of smoothness of that point in upward direction. We are now ready to present a refinement of \cite[Th. 2.9]{SMP24}. 

\begin{theorem}\label{theopresmooth}
	Let $\mathbb{X}$ and $\mathbb{Y}$ be finite-dimensional  normed linear spaces and let $x \in \mathbb{X}$ be $k$-smooth. Suppose that $f_1,f_2,\dots,f_k \in J(x)$ are linearly independent. If  $T\in\mathbb{L}(\mathbb{X},\mathbb{Y})$ preserves Birkhoff-James orthogonality at $x$ with respect to  $\ker f_i,$ for all $1\leq i\leq k, $ then $Tx$ is a  $p$-smooth point of $\mathbb{Y}$  for some $k\leq p.$\
\end{theorem}

\begin{proof}
	Consider the set $F=co(\{f_1, f_2, ..., f_k\}).$ Suppose that $A=\bigcup\limits_{f\in F}\ker(f).$ Clearly, $A\subset x^{\perp_B}.$ Since   $T\in\mathbb{L}(\mathbb{X},\mathbb{Y})$ preserves Birkhoff-James orthogonality at $x$ with respect to  $\ker f_i,$ for all $1\leq i\leq k, $ it follows from Corollary \ref{corpre} that  $T$ preserves Birkhoff-James orthogonality at $x$ with respect to $\ker f,$ for all $f\in co(\{f_1,f_2,\dots,f_k\})$ and so $T(A)\subset {(Tx)}^{\perp_B}.$
	Suppose on the contrary that $Tx$ is a $p$-smooth point for some $p<k.$ Then there exist $p$ linearly independent elements $g_1, g_2, ..., g_p\in J(Tx).$ Consider the set  $G=\Big \{\sum_{i=1}^{p}\beta_ig_i :\sum_{i=1}^{p}\beta_i=1, \beta_i\in \mathbb{R} \,\, \forall\, i =1,2,\ldots,p\Big \}.$  It is easy to observe that  $J(Tx)\subset G.$ Let $B=\bigcup\limits_{g\in G}ker(g)$ and so  $  {(Tx)}^{\perp_B}\subset B.$ Since $T(A)\subset {(Tx)}^{\perp_B},$ clearly, $T(A)\subset B.$ Now, it follows from  Lemma \ref{lem cotoaf} that $T(\mathbb{X})\subset B.$ But for each $g\in G$, $g(Tx)=\| Tx\|$  and so  $Tx\notin B,$ a contradiction. Thus,  $Tx$ is a $p$-smooth point for some $p\geq k.$ This establishes the theorem.
\end{proof}

In the following result, we point out an important geometric attribute of the local preservation of orthogonality with respect to hyperspaces, by an operator on a finite-dimensional  normed linear space.

\begin{theorem}\label{theexttoext}
	Let $\mathbb{X}$ be a finite-dimensional  normed linear space and let $x \in \mathbb{X}$ be non-zero. Let $T\in\mathbb{L}(\mathbb{X})$ be such that  $T(x^{\perp_B})=(Tx)^{\perp_B}.$ Then for each $f\in Ext~J(x),$ there exists $g\in Ext~J(Tx) $ such that $T(\ker f)= \ker g.$ Moreover,  for each $V\in \mathcal{V}(x),$ there exists $W\in \mathcal{V}(Tx)$ such that $T(V)\subset W.$ 
\end{theorem}

\begin{proof}
	Since $T(x^{\perp_B})=(Tx)^{\perp_B}$ and $x\neq 0$, it follows that $Tx\neq 0$ and $T$ is bijective. From the hypothesis, it follows that $T$ preserves Birkhoff-James orthogonality at $x$ with respect to $\ker f,$ for all $f\in Ext~J(x).$ 	Let $Ext~J(x)=\{f_1,f_2,\dots,f_n\}.$ Now, from  \ref{kerf-subset-kerg}, it follows that there exist $g_1,g_2,\dots,g_n \in J(Tx)$ such that $T(\ker f_i)=\ker g_i,$ for all $1\leq i\leq n$ and 
	$T\Big(\bigcup\limits_{f\in F} \ker f\Big)\subset \bigcup\limits_{g\in G} \ker g,$ where $F=co(\{f_1,f_2,\dots,f_n\})$ and $G=co(\{g_1,g_2,\dots,g_n\}).$
	Therefore, \[(Tx)^{\perp_B}=T(x^{\perp_B})=T\Big(\bigcup\limits_{f\in F} \ker f\Big)\subset \bigcup\limits_{g\in G} \ker g\subset (Tx)^{\perp_B}.\] 
	Hence $ \bigcup\limits_{g\in G} \ker g= (Tx)^{\perp_B}.$ Next, we claim that $ Ext~J(Tx)\subset \{g_1,g_2,\dots,g_n\}.$ Contrary to our claim suppose that there exists $h\in Ext~J(Tx)$ such that $ h\notin \{g_1,g_2,\dots,g_n\}.$ This implies that $G\subsetneq J(Tx)$ and so  $ \bigcup\limits_{g\in G} \ker g\subsetneq (Tx)^{\perp_B},$ a contradiction. This establishes our claim that $ Ext~J(Tx)\subset \{g_1,g_2,\dots,g_n\}.$ From Theorem \ref{theopresmooth}, it follows that if $x$ is a $k$-smooth point of $\mathbb{X},$ then $Tx$ is a $p$-smooth point of $\mathbb{X}$ for some $p\geq k.$ This implies that $|Ext~J(x)|\leq |Ext~J(Tx)|$ and so $ Ext~J(Tx)= \{g_1,g_2,\dots,g_n\}.$ Therefore,  for all $f\in Ext~J(x),$ $T(\ker f)= \ker g$ for some $g\in Ext~J(Tx).$
	
	Next, we show that for each $V\in \mathcal{V}(x),$ there exists $W\in \mathcal{V}(Tx)$ such that $T(V)\subset W.$ Let $V$ be the associated cone corresponding 
	to $f_1,f_2\in Ext~J(x),$ i.e., $V=\{y\in \mathbb{X}:f_1(y)\geq 0\text{ and }f_2(y)\leq 0\}.$ From the first part of the proof, it follows that $T(\ker f_1)=\ker g_1$ and $T(\ker f_2)=\ker g_2$ for some $g_1,g_2\in Ext~J(Tx).$ Let $W$ be the associated cone corresponding to $g_1,g_2,$ i.e., $W=\{y\in \mathbb{X}:g_1(y)\geq 0\text{ and }g_2(y)\leq 0\}.$ From Theorem \ref{kerf-subset-kerg}, it follows  that $T(V)\subset W\cup-W.$ Therefore, it follows from Corollary \ref{conetocone} that $T(V)\subset W$ or $T(V)\subset -W.$ This completes the proof.
\end{proof}
We note that the converse of Theorem \ref{theexttoext} may not be true, which is clear from the following example. 
\begin{example}
	Let $\mathbb{X}=\ell_{\infty}^{3}.$ Define $T\in \mathbb{L}(\mathbb{X})$ by \[T(\alpha,\beta,\gamma)=(\alpha,\beta,\beta+\gamma), \text{ for all }(\alpha,\beta,\gamma)\in \mathbb{X}.\]
	Let $x=(1,1,0).$ Then $Tx=(1,1,1).$ Consider $f_1,f_2 \in S_{\mathbb{X}^*}$ given by
		\[f_1(\alpha,\beta,\gamma)=\alpha,~f_2(\alpha,\beta,\gamma)=\beta.\]
	It is easy to observe that $f_1,f_2 \in Ext~J(x)$ and also  $f_1,f_2 \in Ext~J(Tx).$
	Clearly, $\ker f_1= span\{(0,1,0),(0,0,1)\}.$ Let $y\in \ker f_1$ be arbitrary, then $y=(0,b,c)$ for some $b,c\in \mathbb{R}.$ Now, $Ty=(0,b,b+c)$ and so $Ty\in \ker f.$ Since $y\in \ker f_1$ is chosen arbitrarily and $T$ is bijective, it follows that $T(\ker f_1)= \ker f_1.$ Similarly, we can show that $T(\ker f_2)= \ker f_2.$ But $T(x^{\perp_B})\neq(Tx)^{\perp_B}$  as $x\not\perp_B (1,1,-1)$ but $Tx\perp_B T(1,1,-1).$
\end{example}

In \cite[Lemma 2.13]{SMP24}, it was shown that in case of  two-dimensional polyhedral  normed linear spaces, if a non-zero linear operator preserves Birkhoff-James orthogonality at two linearly independent extreme points of the unit ball then the operator must be bijective.
In the following theorem, we show that in case of  $n$-dimensional  normed linear spaces if a linear operator  preserves  Birkhoff-James orthogonality at an $n$-smooth point $x$ then it is necessary for the operator to be bijective unless $Tx=0.$ 

\begin{theorem}\label{th bijective}
	Let $\mathbb{X}$ and $\mathbb{Y}$ be  $n$-dimensional  normed linear spaces. Let a non-zero $x\in \mathbb{X}$ be such that there exist $n$ linearly independent elements $f_1,f_2,\dots,f_n \in J(x).$ If $T\in\mathbb{L}(\mathbb{X},\mathbb{Y})$ preserves Birkhoff-James orthogonality at $x$ with respect to $\ker f_i,$ for all $1\leq i\leq n,$ then either $Tx=0$ or $T$ is bijective.
\end{theorem} 

\begin{proof}
	Let  $T\in\mathbb{L}(\mathbb{X},\mathbb{Y})$ preserves Birkhoff-James orthogonality at $x$ with respect to $\ker f_i,$ for all $1\leq i\leq n.$ Suppose that $Tx\neq 0.$ From Theorem \ref{kerf-subset-kerg}, it follows that there exist $g_1,g_2,\dots,g_n\in J(Tx)$ such that $T(\ker f_i)\subset \ker g_i,$ for all $1\leq i\leq n.$ Since $f_1,f_2,\dots,f_n$ are linearly independent, it follows that $\bigcap\limits_{i=1}^{n}\ker f_i=\{0\}.$ Now, we show that $\ker T=\{0\}.$ Suppose on the contrary that there exists $z(\neq 0)\in \ker T.$ Since $\bigcap\limits_{i=1}^{n}\ker f_i=\{0\},$ it follows that there exists $1\leq k\leq n$ such that $z\notin \ker f_k.$ So $\mathbb{X}=span\{z\}\oplus\ker f_k.$ Let $u\in \mathbb{X}$ be arbitrary. Then $u=\alpha z+h$ for some $\alpha\in\mathbb{R}$ and some $h\in \ker f_k.$ Hence $Tu=Th\in \ker g_k.$ Since $u$ is chosen arbitrarily, it follows that $T(\mathbb{X})\subset \ker g_k.$ This contradicts that $g_k\in J(Tx).$ Thus, $\ker T=\{0\}$ and so $T$ is bijective. This completes the proof.
\end{proof}
The following corollary is an easy consequence of Theorem \ref{th bijective}.
	\begin{cor}\cite[Lemma 2.13]{SMP24}
		Let $ \mathbb{X} $ and $\mathbb{Y} $ be two-dimensional polyhedral  normed linear spaces. If a non-zero operator $T\in\mathbb {L}(\mathbb{X},\mathbb{Y})$ preserves Birkhoff-James orthogonality at any two linearly independent extreme points of $B_\mathbb{X},$ then $T$ is bijective.
\end{cor}
\begin{proof}
	Let $u,v$ be two linearly independent extreme points of $B_\mathbb{X}$ and let us assume that $T\in\mathbb {L}(\mathbb{X},\mathbb{Y})$ preserves Birkhoff-James orthogonality at $u,v.$ Suppose on the contrary that $T$ is not bijective. Since $ \mathbb{X} $ is a two-dimensional polyhedral  normed linear space, $u$ and $v$ are $2$-smooth points of $\mathbb X.$ Therefore, it follows from Theorem \ref{th bijective} that $Tu=Tv=0,$ which implies that $ T $ is the zero operator, a contradiction. Thus, $T$ is bijective.
\end{proof}
We would like to mention that if $\mathbb{X}$ is an $n$-dimensional  normed linear space and if a non-zero $x\in \mathbb{X}$ is such that there exist $p$ linearly independent elements $f_1,f_2,\dots,f_p \in J(x),$ for some $p<n$ and if $T\in\mathbb{L}(\mathbb{X})$ preserves Birkhoff-James orthogonality at $x$ with respect to $\ker f_i,$ for all $1\leq i\leq p$ and  $Tx\neq0,$ then $T$ may not be bijective. We present the following example to establish this claim.

\begin{example}
	Let $\mathbb{X}=\ell_{\infty}^{3}.$ Consider $T\in \mathbb{L}(\mathbb{X})$ given by \[T(\alpha,\beta,\gamma)=(\alpha,\beta,\beta), \text{ for all }(\alpha,\beta,\gamma)\in \mathbb{X}.\]
	Let $x=(1,1,1).$ Then $Tx=(1,1,1).$ Consider $f_1,f_2 \in S_{\mathbb{X}^*}$ given by
		\[f_1(\alpha,\beta,\gamma)=\alpha,~f_2(\alpha,\beta,\gamma)=\beta.\]
	It is easy to observe that $f_1,f_2 \in J(x).$
	Clearly, $\ker f_1= span\{(0,1,0),(0,0,1)\}.$ Let $y\in \ker f_1$ be arbitrary, then $y=(0,b,c)$ for some $b,c\in \mathbb{R}.$ Now, $Ty=(0,b,b)$ and so $Ty\in (Tx)^{\perp_B}.$ Since $y\in \ker f_1$ is chosen arbitrarily , it follows that $T(\ker f_1)\subset(Tx)^{\perp_B} .$ Similarly, we can show that $T(\ker f_2)\subset(Tx)^{\perp_B}.$ Therefore, $T$ preserves Birkhoff-James orthogonality at $x$ with respect to $\ker f_1$ and $\ker f_2$ respectively. However, it is trivially true that $T$ is not bijective.
\end{example}

We next characterize the isometries on a finite-dimensional polyhedral  normed linear space.
\begin{theorem}\label{isometry}
	Let $ \mathbb{X} $ be an $n$-dimensional polyhedral  normed linear space and let $T\in \mathbb{L}(\mathbb{X})$ be of norm one. Then $T$ is an isometry if and only if $T$ satisfies the following conditions: 
	\begin{itemize}
		\item[(i)]  $\| Tu\|=\|Tv\|$ for any $u,v \in Ext~B_\mathbb{X}.$
		\item[(ii)]For each $u\in Ext~B_\mathbb{X},$ there exist $n$ linearly independent elements $f_1,f_2,\dots,f_n \in J(u)$  such that $T\in\mathbb{L}(\mathbb{X})$ preserves Birkhoff-James orthogonality at $u$ with respect to $\ker f_i,$ for all $1\leq i\leq n.$
	\end{itemize} 
\end{theorem}

\begin{proof}
	As the necessary part of the theorem follows trivially, we only prove the sufficient part.
	Suppose that for each $u\in Ext~B_\mathbb{X},$ there exist $n$ linearly independent elements $f_1,f_2,\dots,f_n \in J(u)$  such that $T\in\mathbb{L}(\mathbb{X})$ preserves Birkhoff-James orthogonality at $x$ with respect to $\ker f_i,$ for all $1\leq i\leq n$ and  $\| Tu\|=\|Tv\|$ for any $u,v \in Ext~B_\mathbb{X}.$ From Theorem \ref{th bijective}, we conclude that $T$ is bijective. From Theorem \ref{theopresmooth}, it follows that the image of each extreme point of $B_\mathbb{X} $ is a scalar multiple of some extreme point of $B_\mathbb{X}.$ Since $\mathbb{X}$ is finite-dimensional, $T$ must attains its norm at an  extreme point of $B_\mathbb{X}.$ Also,  we have  $\| Tu\|=\|Tv\|$ for any two extreme points $u,v$ of  $B_\mathbb{X}.$  Thus,  $\| Tu\|=\|T\|=1$ for all extreme points $u$ of $B_\mathbb{X}.$  Since $T$ is bijective, $T^{-1}$ exists and maps each extreme point of $B_\mathbb{X} $ to some extreme point of $B_\mathbb{X}.$ Now, $T^{-1}$ also attains norm at some extreme point of $B_\mathbb{X}$ and so $\|T^{-1}\|=1.$ 
	Therefore, for all $x\in\mathbb{X},$  $\|Tx\|\leq \|x\|=\|T^{-1}(Tx)\|\leq \|Tx\|.$ This implies $\|Tx\|=\|x\|,$ for all  $x\in\mathbb{X}.$  Thus, $T$ is an isometry. 
\end{proof}
In  \cite[Th. 2.19, Th. 2.22]{SMP24},  proper refinements of the Blanco-Koldobsky-Turn\v{s}ek Theorem were obtained for  $ \mathbb{X}=\ell_{\infty}^n,~\ell_{1}^n. $ In fact, in \cite[Th. 2.19, Th. 2.22]{SMP24}, it was shown that if a linear operator on $\mathbb{X}$ preserves Birkhoff-James orthogonality at each extreme point of $B_{\mathbb{X}},$  then T is a scalar multiple of an isometry.  In  \cite[Lemma 2.18, Lemma 2.21]{SMP24}, it was shown that if $T\in \mathbb {L}(\mathbb{X})$ maps  each extreme point of $B_{\mathbb{X}}$ to scalar multiples of some extreme point of $B_{\mathbb{X}},$  then $\| Tu\|=\|Tv\|$ for any $u,v \in Ext~B_\mathbb{X}.$  Therefore, from Theorem \ref{isometry}, we get  further refinement of  the Blanco-Koldobsky-Turn\v{s}ek Theorem in  $ \mathbb{X}=\ell_{\infty}^n,~\ell_{1}^n. $ 
\begin{theorem}
	Let $\mathbb{X}=\ell_{\infty}^n(or,~\ell_1^n).$ If $T\in \mathbb {L}(\mathbb{X})$ preserves Birkhoff-James orthogonality at each extreme point  $x\in B_{\mathbb{X}}$ with respect to  $\ker f,$ for any $n$ linearly independent  $f\in J(x),$ then $T$ is a scalar multiple of an isometry.
	\end{theorem}

The next result characterizes  $\ell_2^n$ out of $\ell_p^n$ spaces using the notion of  $\mathcal K$-set. Recall that a set $A\subset S_{\mathbb{X}}$ is said to be a $\mathcal K$-set \cite{SMP24}   if any operator $T\in \mathbb{L}(\mathbb{X})$ that preserves Birkhoff-James orthogonality at each point of $A,$  is necessarily a scalar multiple of an isometry.

\begin{theorem}
	Let $\mathbb{X}=\ell_p^n, 1 < p < \infty $.  Then the followings are equivalent.
	\begin{itemize}
		\item [(I)]  $\mathbb{X}=\ell_2^n$. 
	\item [(II)] A nonempty set $A\subset S_{\mathbb{X}}$ is a minimal $\mathcal K$-set if and only if $A$  satisfies the  following two conditions:
	\begin{itemize}
		\item[(i)] $A=\{x_1,x_2,\dots,x_n\},$ where $x_1,x_2,\dots,x_n$ are linearly independent.
		\item[(ii)] If $A=A_1\cup A_2,$ where $A_1\neq\phi, $ $A_2\neq\phi,$ then $A_1 \not \perp_B A_2.$
	\end{itemize}
		\end{itemize} 
\end{theorem}
	\begin{proof}
	(I) implies (II) follows from \cite[Th. 2.4]{SMP24}. We prove that (II) implies (I).
		This is equivalent  to proving that  for each $p\neq 2,$ there exists an operator $T\in \mathbb{L}(\mathbb{X})$ such that  $T$ preserves Birkhoff-James orthogonality at some $A \subset S_{\mathbb{X}},$ satisfying (i) and (ii), but $T$ is not  a scalar multiple of an isometry. Note that for $u=(u_1,u_2,\dots,u_n),v=(v_1,v_2,\dots,v_n)\in \mathbb{X},$ the unique semi inner product defined as follows:
	\[[u,v]_p =  \begin{cases}
		\frac{\sum_{i=1}^{n} u_iv_i|v_i|^{p-2}}{\|v\|_{p}^{p-2}},&~v\neq0\\
		0,&~v=0.
	\end{cases}\]
	From \cite[Th. 2]{G67}, it follows that for $u,v\in \mathbb{X},$ $u\perp_B v$ if and only if $[v,u]_p=0.$\\
	Let $p\neq 2.$ We consider the following two cases separately:
	
	\smallskip
	\textbf{Case 1}: Suppose that $n=2.$  Consider $T\in \mathbb{L}(\mathbb{X})$ defined by
	\[T(x,y)=(x+y,x-y),\text{ for all }(x,y)\in \mathbb{X}.\]
	It is easy to verify that  $T$ is not a scalar multiple of an isometry unless $p=2.$  Let $ A = \big\{u,\frac{v}{\|v\|}\big\},$ where $u=(1,0)$ and $v=(1,1).$
	We claim  that $T$ preserves orthogonality on $A.$
	Clearly $u^{\perp_B}=\{(0,y):y\in\mathbb{R}\}$ and $v^{\perp_B}=\{(z,-z):z\in\mathbb{R}\}.$  Then $[T(0,y), T(1,0)]_p=[(y,-y), (1,1)]_p=0,$ for all $y\in \mathbb{R}$ and $[T(z,-z), T(1,1)]_p=[(0,2z), (2,0)]_p=0,$ for all $z\in \mathbb{R}.$
	Hence $T$ preserves Birkhoff-James orthogonality on $A.$ 
Thus,  $A \subset S_{\mathbb{X}}$ satisfies (i) and (ii) and $T$ preserves Birkhoff-James orthogonality at each point of $A$ but $T$ is not a scalar multiple of an isometry.

	\smallskip
	\textbf{Case 2}: Suppose that $n > 2.$ For each $p,$ consider $T\in \mathbb{L}(\mathbb{X}),$ whose matrix representation with respect to the standard ordered basis of $ \mathbb{R}^n $ is given by
	\[\begin{bmatrix} 
		S &0\\
		0 &D
		
	\end{bmatrix},
\text{	where }S=\begin{bmatrix} 
		1 & 1 \\
		1 & -1
	\end{bmatrix},~ D =2^{\frac{1}{q}}I_{n-2},~\frac{1}{p}+\frac{1}{q}=1.\]
	Clearly, $T$ is not a scalar multiple of an isometry as $p\neq 2.$ For each $i\in\{2,\ldots, n\},$ let $e_i=(0,\dots,0,1,0,\ldots, 0)$, where all the coordinates of $e_i$ are zero except the $i$-th coordinate, which is $1$.  Let $A = \Big\{\frac{u}{\|u\|}, e_2,e_3, \ldots, e_n \Big\}, $ where $u=(1,1,\ldots, 1).$
	We claim that $T$ preserves Birkhoff-James orthogonality on $A.$  Observe that  $e_i^{\perp_B}=\{(y_1,y_2,\dots,y_n)\in \mathbb{X}:y_i=0\},$  for each $ 2 \leq i \leq n.$ 
	For  $z=(z_1,0,z_3,\dots, z_n)\in (e_2)^{\perp_B},$ 
	\[[Tz, Te_2]_p=[(z_1,z_1, 2^{\frac{1}{q}}z_3\ldots, 2^{\frac{1}{q}}z_n),  (1,-1, 0, \ldots, 0)]_p=0.\] This shows that $T$ preserves Birkhoff-James orthogonality at $e_2.$ \\
	Next, let $3\leq i\leq n.$ Then for  $w=(w_1,w_2,\dots, w_{i-1},0,w_{i+1},\ldots, w_n)\in (e_i)^{\perp_B},$ 
	\begin{eqnarray*}
		&&[Tw, Te_i]_p\\
		&=&\scalebox{.95}{$[(w_1+w_2,w_1-w_2,2^{\frac{1}{q}}w_3,\dots,2^{\frac{1}{q}}w_{i-1},0,2^{\frac{1}{q}}w_{i+1},\dots, 2^{\frac{1}{q}}w_n), (0,\dots,0,2^{\frac{1}{q}},0 , \ldots, 0)]_p$}\\
		&=&0.
	\end{eqnarray*}
	This shows that $T$ preserves Birkhoff-James orthogonality at $e_i, $ for $3 \leq i \leq n.$ Therefore,  $T$ preserves Birkhoff-James orthogonality at $e_i,$ for $2 \leq i \leq n.$ 
	Finally, we show that $T$  preserves Birkhoff-James orthogonality at $u.$
	Now, $u^{\perp_B}=\{(v_1,v_2,\dots,v_n)\in \mathbb{X}:v_1+v_2+\dots+v_n=0\}.$
	Then for  $v=(v_1,v_2,\dots,v_n)\in (u)^{\perp_B},$ 
	\begin{eqnarray*}
		[Tv,Tu]_p&=&[(v_1+v_2,v_1-v_2,2^{\frac{1}{q}}v_3\dots, 2^{\frac{1}{q}}v_n),(2,0,2^{\frac{1}{q}}\ldots, 2^{\frac{1}{q}})]_p\\
		&=&\frac{(v_1+v_2)2^{p-1}+2^{\frac{1}{q}}v_32^{\frac{p-1}{q}}+2^{\frac{1}{q}}v_42^{\frac{p-1}{q}}+\dots+ 2^{\frac{1}{q}}v_n2^{\frac{p-1}{q}}}{\|Tu\|^{p-2}}\\
		&=&\frac{(v_1+v_2)2^{\frac{p}{q}}+v_32^{\frac{p}{q}}+v_42^{\frac{p}{q}}+\dots+ v_n2^{\frac{p}{q}}}{\|Tu\|^{p-2}} \\ 
		&=&\frac{(v_1+v_2+\dots+ v_n)2^{\frac{p}{q}}}{\|Tu\|^{p-2}}\\
		&=& 0.
	\end{eqnarray*}
Thus,  $T$ preserves Birkhoff-James orthogonality at each point of $A$ but $T$ is not a scalar multiple of an isometry. This completes the proof of the theorem.
	\end{proof}

We end this article with the following conjecture.
\begin{conjecture}
	Let $\mathbb{X}$ be an $n$-dimensional  normed linear space. Then the following two conditions are equivalent. 
		\begin{itemize}
		\item [(I)]  $\mathbb{X} $ is a Hilbert space. 
		\item [(II)] A nonempty set $A\subset S_{\mathbb{X}}$ is a minimal $\mathcal K$-set if and only if $A$  satisfies the  following two conditions:
		\begin{itemize}
			\item[(i)] $A=\{x_1,x_2,\dots,x_n\},$ where $x_1,x_2,\dots,x_n$ are linearly independent.
			\item[(ii)] If $A=A_1\cup A_2,$ where $A_1\neq\phi, $ $A_2\neq\phi,$ then $A_1 \not \perp_B A_2.$
		\end{itemize}
	\end{itemize} 
	\end{conjecture}


\begin{thebibliography}{99}
	\bibitem{B35} Birkhoff, G.,  \textit{Orthogonality in linear metric spaces}, Duke Math. J., \textbf{1} (1935) 169-172.
	
	\bibitem{BT06} Blanco, A. and Turn\v{s}ek, A., \textit{On maps that preserve orthogonality in normed spaces}, Proc. Roy. Soc. Edinburgh Sect. A, \textbf{136}(2006) 709-716.
	
	\bibitem{DMP22} Dey, S., Mal, A. and  Paul, K., \textit{$k$-smoothness on polyhedral Banach space}, Colloq. Math., \textbf{169} (2022) 25-37.
	
	\bibitem{G67} Giles, J.R., \textit{Classes of semi-inner-product spaces}, Trans. Amer. Math. Soc., \textbf{129} (1967) 436–446.
	
	\bibitem{J47} James, R.C., \textit{Orthogonality and linear functionals in normed linear spaces}, Trans.  Amer.  Math. Soc., \textbf{61} (1947) 265-292.
	
	\bibitem{KS05} Khalil, R. and Saleh, A., \textit{Multi-smooth points of finite order}, Missouri J. Math. Sci., \textbf{17} (2005) 76-87.
	
	\bibitem{K93} Koldobsky, A., \textit{Operators preserving orthogonality are isometries}, Proc. R. Soc. Edinburgh Sect. A, \textbf{123} (1993) 835-837.
	
	\bibitem{LR07} Lin, B. and Rao, T.S.S.R.K., \textit{Multismoothness in Banach Spaces}, Int. J. Math. Math. Sc., \textbf{2007} (2007) Art. ID 52382.
	
	\bibitem{Book24}  Mal,  A., Paul, K. and Sain, D., \textit{Birkhoff-James orthogonality and geometry of operator spaces},  Infosys Science Foundation Series, Springer Singapore, 2024. ISBN 978-981-99-7110-7, https://doi.org/10.1007/978-981-99-7111-4.
	
	\bibitem{MPD22}  Mal, A., Dey, S. and Paul, K., \textit{Characterization of k-smoothness of operators defined between infinite-dimensional spaces}, Linear Multilinear Algebra, \textbf{70} (2022) 3477–3489.
	
	\bibitem{MP20} Mal, A. and Paul,  K.,\textit{ Characterization of k-smooth operators between Banach spaces}, Linear Algebra Appl., \textbf{586} (2020) 296-307.
	
	
	\bibitem{PSG16} Paul, K., Sain,  D. and  Ghosh, P., \textit{Birkhoff-James orthogonality and smoothness of bounded linear operators}, Linear Algebra Appl.,\textbf{ 506} (2016) 551-563.
	
	\bibitem{S20} Sain, D., \textit{On the norm attainment set of a bounded linear operator and semi-inner-products in normed spaces},  Indian J. Pure Appl. Math., \textbf{ 51} (2020) 179–186.
	
	\bibitem{S18} Sain, D., \textit{On the norm attainment set of a bounded linear operator}, J. Math. Anal. Appl., \textbf{457} (2018) 67-76.
	
	
	\bibitem{S17}Sain, D., \textit{Birkhoff–James orthogonality of linear operators on finite dimensional Banach spaces}, J.	Math. Anal. Appl., \textbf{447} (2017) 860–866.
	
	\bibitem{SMP20} Sain, D., Mal, A. and Paul, K., \textit{Some remarks on Birkhoff-James orthogonality of linear operators}, Expo. Math., \textbf{38} (2020) 138-147.
	
\bibitem{SMP24}	Sain,  D.,  Manna, J. and  Paul,  K., \textit{ On local preservation of orthogonality and its application to isometries}, Linear Algebra Appl., \textbf{690} (2024) 112-131. 
	
	\bibitem{SPMR20} Sain, D., Paul, K., Mal, A. and  Ray, A.,\textit{ A complete characterization of smoothness in the space of bounded linear operators}, Linear Multilinear Algebra, \textbf{68} (2020) 2484-2494.
	
	\bibitem{SRT21} Sain, D., Roy, S. and Tanaka, R., \textit{Level numbers of a bounded linear operator between normed linear spaces and singular value decomposition revisited}, Linear Multilinear Algebra, \textbf{70} (2022) 3905-3922.
	

	\bibitem{S70} Singer, I., \textit{Best approximation in normed linear spaces by elements of linear subspaces}, Die Grundlehren der mathematischen Wissenschaften, \textbf{171}, Springer-Verlag, Berlin, Heidelberg, New York(1970).
	
	 \bibitem{W18} W\'{o}jcik, P., \textit{$k$-smoothness: an answer to an open problem}, Math. Scand., \textbf{123} (2018) 85-90.
	
\end{thebibliography}
\end{document}